\documentclass{article}
\usepackage{graphicx} % Required for inserting images

\usepackage{amsfonts,amsmath,amsthm,amssymb}
\usepackage{mathtools}
\usepackage[numbers,compress]{natbib}
\usepackage[]{geometry}
\usepackage{xcolor}
\usepackage{enumitem}
\usepackage{hyperref}

\theoremstyle{plain}
\newtheorem{lem}{Lemma}
\newtheorem{prop}[lem]{Proposition}
\newtheorem{assum}[lem]{Assumption}
\newtheorem{eg}[lem]{Example}
\newtheorem{theorem}[lem]{Theorem}
\newtheorem{remark}[lem]{Remark}

\newcommand{\R}{\mathbb{R}}
\newcommand{\Z}{\mathbb{Z}}
\renewcommand{\P}{\mathbb{P}}
\newcommand{\TV}[1]{\|#1\|^{\mathrm{TV}}}

\DeclareMathOperator{\sign}{sign}
\DeclareMathOperator{\Even}{Evn}

\usepackage{etoolbox}
\patchcmd{\bibliography}{volume}{Volume}{}{}
\makeatletter
\renewenvironment{proof}[1][\proofname] {\par\pushQED{\qed}\normalfont\topsep6\p@\@plus6\p@\relax\trivlist\item[\hskip\labelsep\bfseries#1\@addpunct{.}]\ignorespaces}{\popQED\endtrivlist\@endpefalse}
\makeatother

\numberwithin{equation}{section}
\numberwithin{lem}{section}

\usepackage[normalem]{ulem}  % in preamble; normalem keeps \emph as italic

\title{Folding representations of reflected diffusions}
\author{David Itkin\thanks{Department of Statistics, London School of Economics and Political Science, UK (E-mail: \href{mailto:d.itkin@lse.ac.uk}{d.itkin@lse.ac.uk})} \and  Ioannis Karatzas\thanks{Departments of Mathematics and Statistics, Columbia University, USA (E-mail: \href{mailto:ik1@columbia.edu}{ik1@columbia.edu})}}

\begin{document}

\maketitle

\begin{abstract}
    Diffusion processes with reflection, possibly oblique, on the boundary of a given sufficiently regular closed convex domain in Euclidean space, are constructed in a novel way via instantaneous transformations (``foldings”) of suitable unconstrained diffusions. 
\end{abstract}

\paragraph*{Keywords:}  Reflected stochastic differential equations $\cdot$ Reflected Brownian Motion $\cdot$ Folding representations $\cdot$ Semimartingale local time
$\cdot$ Oblique reflection

\paragraph*{MSC 2020 Classification:}
60H10 
$\cdot$ 60J60
$\cdot$ 60J65 
\section{Introduction}
This paper develops a novel approach to the classical problem of constructing solutions $Y_t\,, ~ 0 \le t < \infty $ to stochastic differential equations with values in a specified closed convex domain $D$ of Euclidean space $\R^d$  and  subject to reflecting conditions (possibly oblique) on its boundary $\partial D$, which we assume is of class $C^3_+$. The method proceeds by constructing an appropriate unconstrained diffusion $X_t\,, ~ 0 \le t < \infty $ on the entire space; then by ``folding this diffusion instantaneously” via a suitable mapping $F : \R^d \to D$; and then by ensuring that the folded process $Y_t = F(X_t), ~0 \le t < \infty $ obtained in this manner satisfies the original equation and boundary conditions. 

This novel approach is quite general, and holds considerable promise for the numerical simulation of diffusions with reflection in a way that does not require enforcing boundary conditions at each discretization step. We do not pursue such simulations here, but leave them for further investigation. 

\smallskip 

\noindent \emph{Preview}:  We develop the above methodology in a step-by-step manner, gaining knowledge and insights along the way. Section \ref{sec:one-dimension} treats the one-dimensional case, and deals successively with the non-negative half-line and reflection at the origin (Subsection~\ref{sec:half-line}), then with the unit interval and reflection at both endpoints (Subsection~\ref{sec:unit_interval}). The folding representation we establish takes $X$ to be a Brownian Motion with a carefully selected state-dependent drift.  Subsection~\ref{sec:one-dimension_Z} studies the case in which the stochastic differential equation specifying the scalar reflected diffusion $Y$ is coupled with a multivariate unconstrained diffusion. The positive orthant and the unit hypercube in $\R^d$ are treated respectively in Subsections~\ref{sec:orthant} and \ref{sec:hypercube} of Section~\ref{sec:multidimensional}, using systems of nonlinear ordinary differential equations and Frobenius-type conditions for their solvability. General convex domains with smooth  boundaries and with general, possibly oblique, reflection on them, are discussed in Section~\ref{sec:convex_domain}: there, we deal first with Brownian Motion and normal reflection on the unit ball (Subsection~\ref{sec:unit_ball}), then use this as a ``springboard” for treating oblique reflection on a general smooth, convex domain (Subsection~\ref{sec:convex_oblique}). The construction here decomposes the reflected diffusion $Y$ into two components: a scalar radial process with reflection and a second process, free of reflection terms, that diffuses on the boundary of the domain. Suitable flows of diffeomorphisms play here crucial roles.

\smallskip

\noindent \emph{Historical Overview}:  \citet{feller1952the,feller1954diffusion,feller1957generalized} achieved a complete classification of scalar diffusions according to their scale function, speed measure, and boundary behavior; see also \citet{ito1965diffusion}, particularly Section 5.7.  To the best of our knowledge, the study of reflected diffusions in several dimensions starts with \citet{wentzell1959on}, and proceeds with the seminal formulation and results of \citet{skorokhod1961stochastic,skorokhod1962stochastic} where bounded variation terms representing reflection on the boundary appear for the first time; see also \citet{watanabe1971on} and Section IV.7 in \citet{ikeda1981stochastic}. \citet{stroock1971diffusion} treat diffusions with reflection via suitable martingale problems; \citet{tanaka1979stochastic} extends the work of Skorokhod to general convex domains; \citet{lions1984stochastic} deal with smooth bounded domains using penalization methods, whereas \citet{saisho1987stochastic} extends their results to more general domains and reflection fields.  Brownian Motions with reflection on orthants and/or wedges are studied by \citet{harrison1981reflected,varadhan1985brownian,harrison1987brownian}, among others. Finally, \citet{dupuis1999convex} construct diffusions in convex polyhedral domains with oblique reflection via convex duality and so-called ``Skorokhod maps”, whose properties they develop.

\section{The one-dimensional case} \label{sec:one-dimension}
We start with the one-dimensional case, which is the easiest by far; and deal first with the positive half-line $I = [0,\infty)$ with reflection at the origin, then with the unit interval $I = [0,1]$ with inward reflection at each endpoint.
\subsection{The positive half-line} \label{sec:half-line}
Let us consider a scalar diffusion with reflection, characterized by the Reflected Stochastic Differential Equation (RSDE) 
\begin{equation} \label{eq:RSDE} 
dY_t = b(Y_t)dt + \sigma(Y_t)dW_t + d\Phi_t
\end{equation}
on the positive half line $I = [0,\infty)$,
where $W$ is a standard Brownian Motion.
The term $\Phi =(\Phi_t)_{t \ge 0}$ is the \emph{reflection} term, meaning that it satisfies $\Phi_0 = 0$, is a nonnegative, nondecreasing process, and its associated measure $d\Phi$ is supported on the set $\{t\geq 0: Y_t =0\}$. A solution to the equation \eqref{eq:RSDE} is a pair $(Y,\Phi)$ of processes on some filtered probability space $(\Omega,\mathcal{F}, (\mathcal{F}_t)_{t \geq 0},\P)$, such that the relationship \eqref{eq:RSDE} holds, $Y$ takes values in $I$ almost surely, $\Phi$ satisfies the conditions described above, and all processes $Y,\Phi,W$ in question are adapted to $(\mathcal{F}_t)_{t \geq 0}$. 

The question we wish to tackle is: When can we represent the reflected diffusion $Y$ as a transformation, or \emph{folding}, 
\begin{equation} \label{eq:goal}
    Y = F(X)
\end{equation} of a standard diffusion $X$ (meaning nonreflected) with constant dispersion coefficient, for a suitable \emph{folding function} $F:\R \to I$? 

The most widely studied RSDE of the type \eqref{eq:RSDE} is Reflected Brownian Motion, corresponding to $b \equiv 0$ and $\sigma \equiv 1$, in which case it is known that 
\[d|B|_t = \mathrm{sign}(B_t)dB_t + dL_t^0(B) = dW_t + dL_t^0(B)\]
holds when $B$ is a Brownian Motion. Here, $W = \int_0^\cdot \mathrm{sign}(B_t)dB_t$ is a Brownian Motion constructed from $B$, $L^0_t(B)$ is the Brownian local time accumulated at the origin over $[0,t]$, and $\mathrm{sign}(x) = 1_{(0,\infty)}(x) - 1_{(-\infty,0]}(x)$. 
We see that, with $F(x) = |x|$, the processes $Y = F(B)$ and $\Phi = L^0(B) = \frac{1}{2}L^0(Y)$, together with the Brownian Motion $W$ just constructed, satisfy \eqref{eq:RSDE}. 

We wish to study this question more generally and systematically, beyond the special choices $b \equiv 0$, $\sigma \equiv 1$ and Reflected Brownian Motion, by looking for a suitable folding function $F$ and for a diffusion $X$, such that the process $Y$ in \eqref{eq:RSDE} is given as \eqref{eq:goal}.

\subsubsection{Finding the right folding function}
To tackle this question we restrict attention to Stochastic Differential Equations (SDEs) of the type 
\begin{equation} \label{eq:X_SDE} 
dX_t = \alpha(X_t)dt + dB_t;
\end{equation}
namely, to Brownian Motions with state-dependent drift coefficient $\alpha:\R \to \R$ to be determined.

Next, inspired by the case of reflected Brownian Motion just discussed, we look for a folding function $F$ which is even; i.e., of the form $F(x) = f(|x|)$ for some function $f:I \to I$. The regularity of $f$ that we will need, is for it to be of class $C^1$ with $f'$ locally absolutely continuous with respect to the Lebesgue measure; that is, we look for a function $f$ which belongs to the Sobolev space 
$W^{2,1}_{\mathrm{loc}}(I)$.
The origin $x = 0$ will be the unique point of possible nondifferentiability of $F$, so we enforce the condition $f(0) = 0$ as this will ensure that the reflection term $\Phi$ we will construct is supported on the set $\{t\geq 0:Y_t = 0\}$.

For any function of this type, using the It\^o\,--Tanaka formula \cite[Theorem~VI.1.5]{revuz1999continuous}, we obtain
\begin{equation} \label{eq:Ito_X}
\begin{aligned}
    dF(X_t) & = F'_-(X_t)dX_t + \frac{1}{2}\int_{\R} L_t^a(X) F''(da)\\
    &  = \big(\sign(X_t)f'(|X_t|)\alpha(X_t) + \frac{1}{2}f''(|X_t|)\big)dt + \sign(X_t)f'(|X_t|)dB_t + f'(0)dL_t^0(X),
\end{aligned}
\end{equation}
where $L_t^a(X) = |X_t-a|-|X_0-a| - \int_0^t \mathrm{sign}(X_s-a)dX_s$ denotes the semimartingale local time of $X$ at $a \in \R$, and $F'_-(\text{resp., 
}F'_+)$  the left (resp., right) derivative of $F$. Here, we used $F''(\{0\}) = F'_+(0) - F'_-(0) = 2f'(0)$ to obtain the term involving the local time of $X$ at zero, and the fact that $f'$ is an absolutely continuous function so $F''(da) = f''(|a|)\, da$ holds on $\R\setminus \{0\}$. Indeed, this allowed us to conclude that  
\[\int_{\R\setminus \{0\}} L_t^a(X)F''(da) = \int_\R L_t^a(X)f''(|a|)da = \int_0^t f''(|X_s|)d[X]_s =
\int_0^t f''(|X_s|)ds,\] courtesy of the occupation times formula (see, e.g., \cite[Corollary~VI.1.6]{revuz1999continuous}).

On the other hand, if \eqref{eq:goal} holds, then we can rewrite \eqref{eq:RSDE} as
\begin{align}
\label{eq:dF_halfline}
dF(X_t) & = b\big(F(X_t)\big)dt + \sigma\big(F(X_t)\big)dW_t + d\Phi_t \\
& = b\big(f(|X_t|)\big)dt + \sigma\big(f(|X_t|)\big)\,\mathrm{sign}(X_t)dB_t + d\Phi_t. \nonumber
\end{align}
This is identical to the expression of \eqref{eq:Ito_X}, if the following equations are satisfied:
\begin{align}
    f'(y) & = \sigma\big(f(y)\big), & \text{for } y \geq 0, \label{eq:nonlinear_ODE}\\
     b\big(f(|x|)\big) & = \sign(x)f'(|x|)\alpha(x) + \frac{1}{2}f''(|x|), & \text{for } x \in \R, \label{eq:drift_equation} \\
    \Phi_t & = f'(0)L_t^0(X) , & t\geq 0.\label{eq:Phi_representation}
\end{align}

We conclude (at least formally) that it suffices to find a function $f:I \to I$ belonging to $W^{2,1}_{\mathrm{loc}}(I)$, which satisfies \eqref{eq:nonlinear_ODE} and $f'>0$. Indeed, if this is accomplished, then \eqref{eq:drift_equation} is satisfied with
\begin{equation} \label{eq:alpha_implicit}
    \alpha(x) = \sign(x)\frac{b(f(|x|)) - \frac{1}{2}f''(|x|)}{f'(|x|)}, \qquad \text{for a.e. } x \in \R;
\end{equation} and \eqref{eq:Phi_representation}  provides then a representation for the reflection term. 

As such, we focus our attention on finding a solution to the nonlinear Ordinary Differential Equation (ODE) of \eqref{eq:nonlinear_ODE} with initial condition $f(0) =0$. 
The following result guarantees the existence of such an $f$, which also satisfies additional desirable properties under appropriate conditions on $\sigma$. In particular, we will need to assume that $\sigma$ is a locally absolutely continuous function, so that the representation
\begin{equation} \label{eq:sigma_abs_cont}
    \sigma(y) = \sigma(0) + \int_0^y \sigma'(u)du
\end{equation}holds for all $y \geq 0$ and some locally integrable $\sigma':I \to \R$.   
\begin{lem} \label{lem:ODE}
    Let $\sigma:I \to (0,\infty)$ be a locally absolutely continuous function with $\int_0^\infty \frac{1}{\sigma(u)}du = \infty$.  
    Then the unique solution $f$ to equation \eqref{eq:nonlinear_ODE} with initial condition $f(0) = 0$ is given by the inverse function $f = g^{-1}$ of 
     \begin{equation} \label{eq:g}
        g(\xi) := \int_0^\xi \frac{1}{\sigma(u)}du, \qquad \forall\,  \xi \in I.
    \end{equation}
 Moreover, this solution satisfies 
    \begin{enumerate}[label = (\roman*),noitemsep]
        \item \label{item:f} $f \geq 0$ and $f\in W^{2,1}_{\mathrm{loc}}(I)$;
        \item \label{item:f'} $f$ is strictly increasing with $f'(0) = \sigma(0)$;
        \item \label{item:f''} $f''(y) = \sigma'(f(y))f'(y)$ for a.e.\ $y \in I$;
        \item \label{item:f_lim} $\lim_{y \to \infty} f(y) = \infty$.
    \end{enumerate}
\end{lem}
\begin{proof}
    Under the stated conditions on $\sigma(\cdot)$, the separation of variables technique shows that the unique solution $f$ to \eqref{eq:nonlinear_ODE} with initial condition $f(0) = 0$, is given by the inverse $f = g^{-1}$ of the continuous, strictly increasing function $g$ of equation \eqref{eq:g}. 
    Since $\sigma$ is positive, we see that $f'(y) = \sigma(f(y)) > 0$. Sending $y \downarrow 0$ yields the expression $f'(0) = \sigma(f(0)) = \sigma(0)$, which establishes item \ref{item:f'}. Moreover, the composition of a locally absolutely continuous function (such as $\sigma$ here) with an increasing function (such as $f$ here) is again locally absolutely continuous (see, e.g., \cite[Corollary~3.65]{leoni2009a}), and the chain rule holds, establishing items \ref{item:f} and \ref{item:f''}. Finally, item \ref{item:f_lim} follows from the representation $f = g^{-1}$ and the fact that $\lim_{\xi \to \infty} g(\xi) = \int_0^\infty \frac{1}{\sigma(u)}du = \infty$ by assumption.
    \end{proof}
 Substituting the expressions for $f'$ in \eqref{eq:nonlinear_ODE}, and for $f''$ in Lemma~\ref{lem:ODE}\ref{item:f''}, into the formula for the drift function $\alpha$ given by \eqref{eq:alpha_implicit}, yields
\begin{equation} \label{eq:alpha_explicit}
    \alpha(x) = \mathrm{sign}(x)\,\zeta\big(f(|x|)\big) \qquad \text{for a.e. } x \in \R,
\end{equation} 
where
\begin{equation} \label{eq:zeta_halfline}
    \zeta (y) := \frac{b(y)}{\sigma(y)} - \frac{1}{2}\sigma'(y), \qquad  y \in I.
\end{equation}
We are now ready to state the main result of this section.  

\begin{theorem} \label{thm:main_half_line}
    Consider functions $b:I \to \R$ measurable, and $\sigma:I \to (0,\infty)$ satisfying the assumptions of Lemma~\ref{lem:ODE}. Suppose also that the function $\zeta$ of \eqref{eq:zeta_halfline} is bounded. Then,
    \begin{enumerate}[label = (\roman*),noitemsep]
        \item the function $f:I \to I$ defined by $f(y) = g^{-1}(y)$ for $y \in I$, with $g$ given by equation \eqref{eq:g}, belongs to $W^{2,1}_{\mathrm{loc}}(I)$, and is the unique solution to the ODE \eqref{eq:nonlinear_ODE} satisfying the initial condition $f(0) = 0$; 
        \item  \label{item:SDE_halfline}the SDE \eqref{eq:X_SDE},
        with drift function $\alpha(\cdot)$ given by \eqref{eq:alpha_explicit},
        has a pathwise unique, strong solution $X$ for every initial value $X_0 \in \R$; and, 
        \item  with $X$ the solution process in \ref{item:SDE_halfline}, the process $Y = f(|X|)$ satisfies the RSDE \eqref{eq:RSDE} with initial condition $Y_0 = f(|X_0|)$, nondecreasing reflection process $\Phi = \sigma(0)L^0(X)$, and  Brownian Motion $W = \int_0^\cdot \mathrm{sign}(X_t)dB_t$. 
    \end{enumerate}  
\end{theorem}
\begin{proof}
    The first claim is just a restatement of Lemma~\ref{lem:ODE}. For the second claim, the boundedness assumption on $\zeta(\cdot)$ ensures that the function $\alpha(\cdot)$ given by \eqref{eq:alpha_explicit} is bounded. Hence, \cite[Proposition~5.5.17]{karatzas1998brownian} ensures that the SDE \eqref{eq:X_SDE} has a pathwise unique, strong solution. Finally, an application of the It\^o\,--Tanaka formula as in \eqref{eq:Ito_X}, together with the expressions for $f'$ and $f''$ in Lemma~\ref{lem:ODE}, establishes the third claim. Note that, since $f(0) = 0$, the measure $d\Phi$ is indeed supported on the set $\{t\geq 0:Y_t = 0\}$, as required.
\end{proof}
\begin{remark} \label{rem:bounded_drift}
    Inspecting the proof, we see that the requirement that $\zeta$ of \eqref{eq:zeta_halfline} be a bounded function, can be replaced by any other condition which ensures global existence of a strong solution to the SDE \eqref{eq:X_SDE} with drift coefficient $\alpha$ given by \eqref{eq:alpha_explicit}.
\end{remark}
\begin{remark} 
    If one wants to match a particular initial condition $Y_0 = y \in I$, then the SDE \eqref{eq:X_SDE}, with drift function $\alpha(\cdot)$ given by \eqref{eq:alpha_explicit}, should be initiated at $X_0 = g(y)$ or $X_0 = - g(y)$. 
\end{remark}

We conclude this subsection with some examples. 
\begin{eg} \label{eg:half-line}
    \begin{enumerate}
        \item The case $\sigma(y) = 1$ and $b(y) = \beta \in \R$ corresponds to Reflected Brownian Motion with drift leading to $f(y) = y$, $F(x) = |x|$, and the SDE \eqref{eq:X_SDE} becomes
        \[dX_t = \beta \,\mathrm{sign}(X_t)dt + dB_t.\] 
        This diffusion, known as \emph{Brownian Motion with Bang--Bang Drift}, is studied in detail in \cite[Section~6.5]{karatzas1998brownian}.
        \item The case $\sigma(y) = 1+y$ and $b(y) = 0$ leads to the solution $f(y) = e^{y}-1$, $F(x) = e^{|x|}-1$, and the SDE \eqref{eq:X_SDE} becomes
        \[dX_t = - \tfrac{1}{2}\mathrm{sign}(X_t)dt + dB_t. \]
        This is the same SDE as in item 1 when $\beta = -1/2$, although the folding function $f$ and, consequently, the reflected diffusion $Y$ differ.
        \item The case $\sigma(y) = \sqrt{1+y^2}$ and $b(y) = 0$ leads to the solution $f(y) = \sinh(y)$, $F(x) = \sinh(|x|)$, and the SDE \eqref{eq:X_SDE} becomes
        \[dX_t = - \tfrac{1}{2}\mathrm{sign}(X_t)\tanh(|X_t|)dt + dB_t.\]
        \end{enumerate}
\end{eg}

\subsection{The unit interval} \label{sec:unit_interval}

Here we study an analogous problem when $Y$ satisfies \eqref{eq:RSDE} on the unit interval $I = [0,1]$. It is clear that if this case is handled, then by a simple scaling and translation argument we can handle any nonempty and nondegenerate bounded interval $[a,b]$. The main difference in this new setup, from the one of the previous subsection, is the existence of the second reflecting boundary point at $y=1$. Accordingly, we adjust the requirements on the reflection term $\Phi$ in \eqref{eq:RSDE} by requiring that it be continuous, of finite variation on compact intervals, carried on the set $\{t\geq 0:Y_t \in \{0,1\}\}$, and inward-pointing (i.e., $\int_0^\cdot 1_{\{Y_t = 0\}}d\Phi_t$ is increasing and $\int_0^\cdot 1_{\{Y_t = 1\}}d\Phi_t$ is decreasing). 

Unlike the positive half-line case, where defining $F$ as the even extension of $f$ was natural, the structure of a candidate function $F$ is less obvious here, due to the second boundary point at $y=1$.

\subsubsection{Reflected Brownian Motion} \label{sec:unit_interval_RBM}

To gain some inspiration, we consider first the case of Reflected Brownian Motion on the unit interval, corresponding again to $b \equiv 0$ and $\sigma \equiv 1$. We start by introducing the folding function
\begin{equation} \label{eq:unit_interval_F}
\begin{aligned} F(x)
& = |x-\Even(x)|, \qquad x \in \R,
\end{aligned}
\end{equation}
where $\Even(x)$ returns the closest even integer to the real number $x$; that is,
$\Even(x) = 2 \lfloor \frac{x+1}{2}\rfloor.$ Loosely speaking, the function $F$ in \eqref{eq:unit_interval_F} will allow us to implement Lord Kelvin's \emph{method of images} in this new context.

We  demonstrate now that $Y = F(B)$, where $B$ is Brownian Motion, leads to a weak solution to \eqref{eq:RSDE} with $b \equiv 0$ and $\sigma \equiv 1$.
The function $F(\cdot)$ is a triangular wave with amplitude one and period two (see Figure~\ref{fig:exampleF} below). Since $\Even(x)$ is constant on each interval $(n,n+1)$, $n \in \Z$, it is easy to verify that $F \in C^2(\R\setminus \mathbb{Z})$, $0 \leq F(x) \leq 1$, 
\[F'_-(x) = \sign\big(x-\Even(x)\big), \qquad x \in \R,\]
and $F''(x) = 0$ for all $x \in \R\setminus \mathbb{Z}$. 
Applying the It\^o\,--Tanaka formula to $F$ yields
\begin{align*} 
dF(B_t) & = F'_-(B_t)dB_t + \tfrac{1}{2}F''(B_t)dt + \tfrac{1}{2}\sum_{n \in \Z}\big(F'_+(n) - F'_-(n)\big)dL_t^n(B) \\
& = \sign\big(B_t-\Even(B_t)\big)dB_t + \sum_{n \in \Z} \big(dL_t^{2n}(B) - dL_t^{2n+1}(B)\big).
\end{align*}
For each $t \ge 0$, the infinite sum contains only finitely many non-zero terms, since $\sup_{s \in [0,t]} |B_s| < \infty$ holds for a.e.\ Brownian path. 

By comparing terms, we see that the process $Y = F(B)$ satisfies the equation \eqref{eq:RSDE} with $b(\cdot) \equiv 0$, $\sigma(\cdot) \equiv 1$,  standard Brownian Motion $W = \int_0^\cdot \sign(B_t-\Even(B_t))dB_t$, and reflection term 
\[\Phi= \sum_{n \in \Z} \big(L^{2n}(B) - L^{2n+1}(B)\big) = \tfrac{1}{2}L^0(Y) - \tfrac{1}{2}L^1(Y).\]
Here the second equality is readily deduced by applying Tanaka's formula to the processes $|Y|$ and $|Y-1|$ respectively, then matching the obtained local time terms of $Y$ to the reflection terms obtained using the folding function $F(\cdot)$.
  Note that this reflection process $\Phi$ is of finite first variation on compact intervals (being the difference of two nondecreasing, continuous and adapted processes) and, in particular, inward-pointing and supported on $\{t \ge 0: Y_t \in \{0,1\}\}$, as required.  
\subsubsection{General coefficients}
We now return to the problem for general drift $b$ and dispersion $\sigma$. Inspired by the above, and by the constructions in Subsection~\ref{sec:unit_interval_RBM}, we look for a folding function $F:\R \to I=[0,1]$ of the form
\[
F(x) = f\big(|x-\Even(x)|\big)
\]
with some function $f:I \to I$ to be determined. As before, we require $f(0) = 0$, but here we need to impose $f(1) = 1$ to ensure the range of $F$ is $[0,1]$. This additional constraint necessitates an extra degree of freedom. Hence, by analogy with \eqref{eq:X_SDE}, we will now look for diffusions of the type 
\begin{equation} \label{eq:X_SDE_nu}
    dX_t = \alpha(X_t)dt + \nu dB_t,
\end{equation}
once again Brownian Motions with state-dependent drift, where now both the drift function $\alpha:\R \to \R$ and the constant dispersion coefficient $\nu > 0$ need to be determined.
Applying the It\^o\,--Tanaka formula to the process $F(X)$ gives
\begin{align}
   \begin{aligned}
   dF(X_t) & = F'_-(X_t)dX_t + \tfrac{1}{2}\int_{\R} L_t^a(X) F''(da) \\
    & = \big(\sign(\overline X_t)f'(|\overline X_t|)\alpha(X_t) + \tfrac{1}{2}\nu^2f''(|\overline X_t|)\big)dt  +\nu \sign(\overline X_t)f'(|\overline X_t|)dB_t  \\
    & \qquad + \sum_{n \in \Z} \big(f'(0)dL_t^{2n}(X) - f'(1)dL_t^{2n+1}(X)\big), \label{eq:Ito_tanaka_interval}
    \end{aligned}
\end{align}
where we set $\overline X_t = X_t - \Even(X_t)$.
As in Subsection~\ref{sec:half-line}, the key equation that needs to be satisfied is obtained by comparing the dispersion coefficients in \eqref{eq:Ito_tanaka_interval} and in the analogue of \eqref{eq:dF_halfline} with Brownian Motion $W = \int_0^\cdot \sign(\overline X_t)dB_t$. Here, this comparison leads to the relationship 
\[\nu f'(y) = \sigma\big(f(y)\big).\] Thus, we need to find a function $f:I \to I$ and a constant $\nu >0$ such that the following hold:
\begin{equation} \label{eq:nonlinear_ODE_nu}
\begin{cases} 
f'(y) = \frac{1}{\nu} \sigma(f(y)), & y \in I,  \\
f(0) = 0, \quad f(1) = 1.
\end{cases}
\end{equation}
This is possible to do, as the next lemma shows.
\begin{lem} \label{lem:ODE_nu}
    Consider an absolutely continuous function $\sigma: I \to (0,\infty)$ as in \eqref{eq:sigma_abs_cont}, and the strictly increasing function $g:I \to [0,\infty)$ of \eqref{eq:g}. Then there exists a unique pair $(f,\nu)$, with $f:I \to I$ of class $C^1(I)$ and $\nu \in (0,\infty)$, which satisfies \eqref{eq:nonlinear_ODE_nu} as well as 
    \begin{enumerate}[label = (\roman*),noitemsep]
        \item \label{item:f_interval} $f(y) = g^{-1}(y/\nu)$ with $\nu =  (\int_0^1 \frac{1}{\sigma(u)}du)^{-1}$;
        \item \label{item:f''_interval} $f \in W^{2,1}(I)$ with $f''(y) = \frac{1}{\nu}\sigma'(f(y))f'(y)$ for a.e.\ $y \in I$;
        \item \label{item:f'_interval} $f'(y) > 0$ for $y \in I$, $f'(0) = \frac{\sigma(0)}{\nu}$ and $f'(1) =\frac{\sigma(1)}{\nu}$.
    \end{enumerate}
\end{lem}
\begin{proof}
    We first consider the equation \eqref{eq:nonlinear_ODE_nu} for fixed $\nu > 0$ and with only the boundary condition $f(0) = 0$ enforced. Akin to the proof of Lemma~\ref{lem:ODE}, this can be solved explicitly, giving rise to the solution $f(y) = h^{-1}(y)$, where
    \[h(\xi) = \int_0^\xi \frac{\nu}{\sigma(u)}du = \nu g(\xi), \qquad \xi \in I, \]
    in terms of the function in \eqref{eq:g}.
    Since $h^{-1}(y) = g^{-1}(y/\nu)$ we obtain the claim $f(y) = g^{-1}(y/\nu)$ in \ref{item:f_interval}. Now dividing both sides of the ODE \eqref{eq:nonlinear_ODE_nu} by $\sigma(f(y))$ and integrating over $[0,1]$ gives
    \[\frac{1}{\nu} = \int_0^1 \frac{f'(y)}{\sigma(f(y))}dy = \int_0^{f(1)} \frac{1}{\sigma(u)}du,\]
    via the change of variables $u = f(y)$ in the final equality. Imposing the condition $f(1) = 1$ leads to the expression for the constant $\nu$ given in  item \ref{item:f_interval}. 
     Proceeding in the same way as in the proof of Lemma~\ref{lem:ODE} yields the formulas in items \ref{item:f''_interval} and \ref{item:f'_interval} of this lemma, and completes the proof.
    \end{proof}
With this lemma proved, we are able to establish the representation \eqref{eq:goal}.

\begin{theorem} \label{thm:main_bounded_interval}
    Consider functions $b:I \to \R$ measurable, and $\sigma:I \to (0,\infty)$ satisfying the assumptions of Lemma~\ref{lem:ODE_nu}. Further, suppose that the function $\zeta$ of \eqref{eq:zeta_halfline} is bounded on $I$. Then, \begin{enumerate}[label = (\roman*),noitemsep]
        \item \label{item:f_nu_interval} the unique global solution $(f,\nu)$, with $f:I \to I$ belonging to  $W^{2,1}(I)$ and $\nu \in (0,\infty)$, to the equation \eqref{eq:nonlinear_ODE_nu}, is given by
        \begin{equation} \label{eq:nu}
           f(y) = g^{-1}\bigg(\frac{y}{\nu}\bigg) \quad \text{for } y \in I, \qquad   \nu = \bigg(\int_0^1 \frac{1}{\sigma(u)}du\bigg)^{-1},
        \end{equation}
         where $g$ is the function in \eqref{eq:g};
        
        \item \label{item:SDE_interval} the SDE \eqref{eq:X_SDE_nu},
        with $\nu$ given by \eqref{eq:nu}, with  
        \begin{equation} \label{eq:alpha_interval}
          \alpha(x) = \nu\sign\big(x-\Even(x)\big)\,\zeta\Big(f\big(|x-\Even(x)|\big)\Big),  \qquad \text{for a.e.~}  x \in \R,
        \end{equation}
        and with $\zeta$ as in \eqref{eq:zeta_halfline},
        has a pathwise unique, strong solution for every initial value $X_0 \in \R$;
        \item the process $Y = f(|X-\Even(X)|)$ satisfies the RSDE \eqref{eq:RSDE} with initial condition $Y_0 = f(|X_0-\Even(X_0)|)$, finite variation process
        \begin{equation} \label{eq:Phi_interval} \Phi = \frac{1}{\nu}\sum_{n \in \Z}\big(\sigma(0)L^{2n}(X) - \sigma(1)L^{2n+1}(X)\big),
        \end{equation} and Brownian Motion $W = \int_0^\cdot \mathrm{sign}(X_t-\Even(X_t))dB_t$. 
    \end{enumerate}  
    \end{theorem}
    \begin{proof} The proof proceeds along essentially the same lines as that of Theorem~\ref{thm:main_half_line}. 
        Item \ref{item:f_nu_interval} is just a restatement of Lemma~\ref{lem:ODE_nu}. For item \ref{item:SDE_interval}, we note that the assumptions on the function $\zeta(\cdot)$ guarantee that $\alpha(\cdot)$ is bounded, so that \eqref{eq:X_SDE_nu} has a unique strong solution, courtesy of \cite[Proposition~5.5.17]{karatzas1998brownian}.
        
        Next, with the function $f:I \to I$ as in \ref{item:f_nu_interval} and $F$ as in \eqref{eq:unit_interval_F}, we apply the It\^o\,--Tanaka formula to $Y = F(X)$ as in equation \eqref{eq:Ito_tanaka_interval}. Courtesy of the fact that $f(\cdot)$ satisfies \eqref{eq:nonlinear_ODE_nu} and $\alpha(\cdot)$ is given by \eqref{eq:alpha_interval}, this shows that $Y = F(X)$ has the dynamics of \eqref{eq:RSDE} with $b(\cdot)$ as its drift coefficient, $\sigma(\cdot)$ as its diffusion coefficient, is driven by the Brownian Motion $W$, and its reflection term equals $\Phi$ as in \eqref{eq:Phi_interval}. Since $F(2n) = 0$ and $F(2n+1) = 1$, for every $n \in \Z$, the process of finite variation on compact intervals $\Phi$ is inward-pointing, and its associated signed measure $d\Phi$ is carried by the set $\{t\geq 0: Y_t \in \{0,1\}\}$ as required.
        \end{proof}
\begin{remark}
    It is easy to see that the construction in Theorem~\ref{thm:main_bounded_interval} continues to work even if $\sigma(0) = 0$ or $\sigma(1) = 0$, as long as $\nu$ defined in \eqref{eq:nu} remains positive. In this case, the representation \eqref{eq:Phi_interval} makes clear that the component of the reflection term $\Phi$, which is supported on the boundary point where $\sigma$ is zero, vanishes.
\end{remark}
We conclude this section with some examples.
 \begin{figure}
        \centering
        \includegraphics[width=0.5\linewidth]{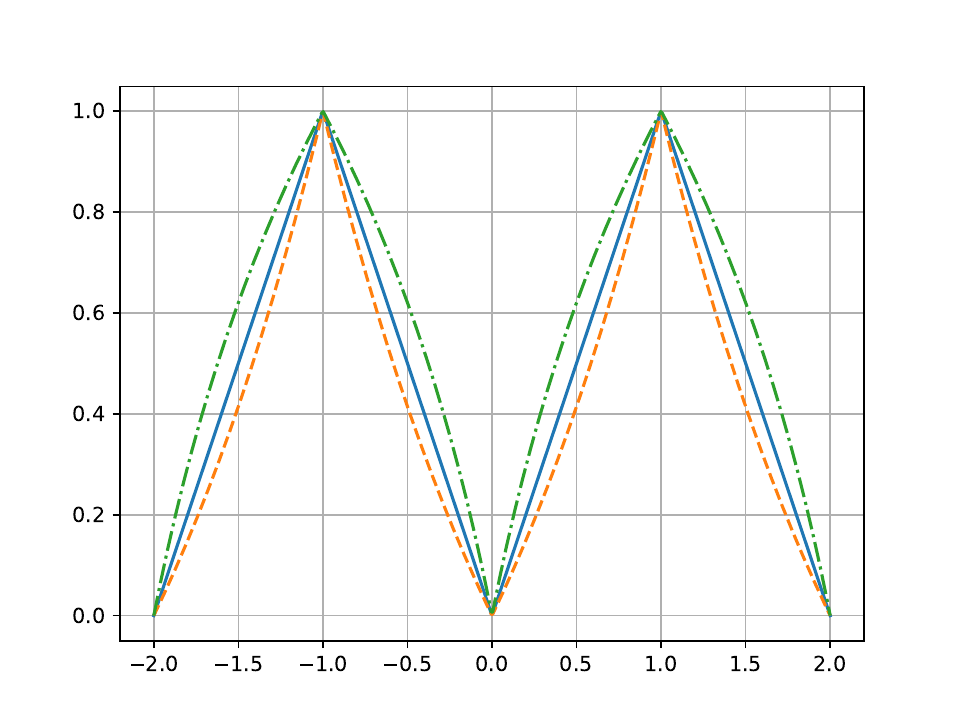}
        \caption{Depicted is $F(x)$ for $x \in [-2,2]$ for  $F(x) = |x-\Even(x)|$ (blue, solid), $F(x) = 2^{|x-\Even(x)|}-1$ (orange, dashed), and $F(x) = \log((e-1)|x-\Even(x)|+1)$ (green, dash-dot).}
        \label{fig:exampleF}
    \end{figure}
\begin{eg}
    \begin{enumerate}[label = (\roman*),noitemsep]
        \item The case $\sigma(y) = 1$ and $b(y) = \beta \in \R$ leads to Reflected Brownian Motion with drift. Here $\nu = 1$, $f(y) = y$, $F(x) = |x - \Even(x)|$, and the SDE \eqref{eq:X_SDE_nu} becomes
         \[dX_t = \beta \,\mathrm{sign}\big(X_t-\Even(X_t)\big)dt + dB_t.\] 
         \item The case $\sigma(y) = 1+y$ and $b(y) = 0$ leads to $\nu = 1/\log 2$, $f(y) = 2^y-1$ and  $F(x) = 2^{|x-\Even(x)|}-1$, whereas the SDE \eqref{eq:X_SDE_nu} becomes
         \[dX_t = -\frac{1}{2\log 2}\sign\big(X_t-\Even(X_t)\big)dt + \frac{1}{\log 2}dB_t.\]
                  \item The case $\sigma(y) = e^{-y}$ and $b(y)  = 0$ leads to $\nu = 1/(e-1)$, $f(y) =\log((e-1)y+1)$ and $F(x) = \log((e-1)|x-\Even(x)|+1)$, whereas the SDE \eqref{eq:X_SDE_nu} becomes
                  \[dX_t =\frac{\sign(X_t - \Even(X_t))}{2(e-1)^2|X_t - \Even(X_t)|+e-1}dt + \frac{1}{e-1}dB_t. \]
    \end{enumerate}
    Figure~\ref{fig:exampleF} plots the three expressions for $F(x)$ obtained above.   
\end{eg}

\subsection{Dependence on a diffusion}\label{sec:one-dimension_Z}

We conclude the one-dimensional case by studying a reflected diffusion $Y$ with values in the positive half-line $I = [0,\infty)$, and coupled with an unconstrained $d$-dimensional diffusion $Z = (Z^1,\dots,Z^d)$ taking values in a closed, bounded subset $E \subset \R^d$. This more general situation will play a pivotal role in Section~\ref{sec:convex_domain} to come.

We start by stating the following assumption. 
\begin{assum} \label{ass:coefficients_Z}
    We fix $d,m \in \mathbb{N}$ and assume coefficients of the following regularity are given:
    \begin{enumerate}[noitemsep,label = (\roman*)]
        \item $b: I \times E \to \R$ and $b_Z:I \times E \to \R^d$ are measurable and bounded,
            \item $\sigma: I \times E \to (0,\infty)$ is bounded, Lipschitz continuous,
            and $\sigma(y,z) \geq \epsilon > 0$ for every $(y,z) \in I \times E$ and some $\epsilon > 0$,
        \item $\sigma_Z:I \times E \to \R^{d \times m}$ is  Lipschitz continuous and $\sigma_Z\sigma_Z^\top$ is uniformly elliptic, 
        \item $\rho:I \times E \to \R^m$ is  Lipschitz continuous and $\|\rho(y,z)\| \le \delta < 1$ for all $(y,z) \in I \times E$ and some $\delta \in (0,1)$,
        \item The function $\widehat \sigma:I \times E \to (0,\infty)$ given by \begin{equation} \label{eq:sigma_hat}
    \widehat \sigma (y,z) = \sigma(y,z)\sqrt{1-\|\rho(y,z)\|^2}
\end{equation}
is such that $\widehat \sigma(y,\cdot)$ and its derivatives $\partial_z \widehat \sigma(y,\cdot)$, $\partial_{zz}\widehat \sigma(y,\cdot)$ are bounded uniformly in $y \in I$.
    \end{enumerate}
\end{assum} 
We study the RSDE system
\begin{equation} \label{eq:RSDE_YZ}
\begin{aligned}
    dY_t & = b(Y_t,Z_t)dt   +  \sigma(Y_t,Z_t)dW_t + d\Phi_t,   \\
    dZ_t & = b_Z(Y_t,Z_t) dt   + \sigma_Z(Y_t,Z_t)d\overline W_t, \\
    d[W,\overline W]_t & = \rho(Y_t,Z_t)dt.  
\end{aligned}
\end{equation}
Here, $W$ is a standard scalar Brownian Motion, $\overline W$ is a standard $m$-dimensional Brownian Motion (with independent components), and the coefficient $\rho(\cdot,\cdot)$ specifies the state-dependent instantaneous correlations between $W$ and $\overline W$. As in Subsection~\ref{sec:half-line}, the reflection term $\Phi$ satisfies $\Phi_0 = 0$, is a nonnegative, nondecreasing process, and its associated measure $d\Phi$ is supported on the set $\{t\geq 0: Y_t =0\}$.

We seek again a folding representation for $Y$, the only process in \eqref{eq:RSDE_YZ} whose dynamics contain a reflection term. Since $Y$ is coupled with $Z$, one can no longer expect a folding representation of the type \eqref{eq:goal} to hold. Instead, here we are looking for a folding function $F:\R \times E \to I$ and an unconstrained real-valued diffusion $X$, such that $Y = F(X,Z)$ holds $\P$-a.s. 

Clearly, this diffusion $X$ will have to be coupled with $Z$, which will take our search for a suitable process $X$ beyond the class \eqref{eq:X_SDE} of Brownian Motion with state-dependent drift. Instead, we will look for a process with dynamics of the form
\begin{equation} \label{eq:dX_Z}
    dX_t = \alpha(X_t,Z_t)dt + dB_t + \sigma_X(X_t,Z_t)^\top d\overline W_t,
\end{equation}
where $B$ is a scalar Brownian Motion independent of $\overline W$, and the coefficients $\alpha:\R\times E \to \R$ and $\sigma_X:\R \times E \to \R^m$ are to be determined. 

By analogy with the approach taken in Subsection~\ref{sec:half-line}, we look here for a folding function of the form $F(x,z) = f(|x|,z)$ for a sufficiently regular function $f:I \times E \to I$. Applying the It\^o\,--Tanaka formula to $F(X,Z)$ leads to
\begin{align}
    dF(X_t,Z_t) & = \sign(X_t)\partial_y f(|X_t|,Z_t)dX_t + \tfrac{1}{2}\partial_{yy}f(|X_t|,Z_t)d[X]_t + \partial_y f(0,Z_t)dL_t^0(X) \nonumber \\
    & \quad + \nabla_z f(|X_t|,Z_t)^\top dZ_t + \tfrac{1}{2}\mathrm{Tr}\big(\nabla_z^2 f(|X_t|,Z_t)d[Z]_t\big) + \partial_y\nabla_zf(|X_t|,Z_t)^\top d[X,Z]_t \nonumber \\
    &  = \bigg(\sign(X_t)\partial_y f(|X_t|,Z_t)\alpha(X_t,Z_t) + \tfrac{1}{2}\partial_{yy}f(|X_t|,Z_t)\big(1+\|\sigma_X(X_t,Z_t)\|^2\big) \nonumber \\
    & \qquad + \nabla_z f(|X_t|,Z_t)^\top b_Z\big(f(|X_t|,Z_t),Z_t\big) + \tfrac{1}{2}\mathrm{Tr}\Big(\nabla_z^2f(|X_t|,Z_t)a_Z\big(f(|X_t|,Z_t),Z_t\big)\Big) \nonumber  \\
    & \qquad + \partial_y \nabla_z f(|X_t|,Z_t)^\top \sigma_Z\big(f(|X_t|,Z_t),Z_t\big)\sigma_X(X_t,Z_t)\bigg)dt \nonumber \\
        & \quad +\Big(\sign(X_t)\partial_yf(|X_t|,Z_t)\sigma_X(X_t,Z_t) + \sigma_Z\big(f(|X_t|,Z_t),Z_t\big)^\top \nabla_z f(|X_t|,Z_t)\Big)^\top d\overline W_t \nonumber \\
    & \quad + \sign(X_t)\partial_yf(|X_t|,Z_t)dB_t + \partial_y f(0,Z_t)dL_t^0(X), \label{eq:dF_Z} 
\end{align}
where we substituted $f(|X_t|,Z_t)$ for $Y_t$, $\partial_y$ refers to derivatives in the first argument of $f$, $\nabla_z$ refers to derivatives in the second argument of $f$, and  we set $a_Z = \sigma_Z\sigma_Z^\top$. To ensure the drift of the process $F(X,Z)$ is equal to $b$, we see that the drift $\alpha$ in \eqref{eq:dX_Z} should be chosen via
\begin{equation} \label{eq:alphaZ}
\alpha(x,z) =\sign(x)\, \zeta_1(x,z), \qquad (x,z) \in \R \times E,
\end{equation}
where for $\nu > 0$ we define
\begin{equation} \label{eq:zetaZ}
\begin{aligned} \zeta_\nu(x& ,z) \\
& = \frac{1}{\partial_y f(|x|,z)}\bigg(b\big(f(|x|,z),z\big) - \tfrac{1}{2}\nu^2\partial_{yy} f(|x|,z)\big(1+\|\sigma_X(x,z)\|^2\big) -\nabla_z f(|x|,z)^\top b_Z\big(f(|x|,z),z\big) \\ 
& \qquad \qquad \qquad - \tfrac{1}{2}\mathrm{Tr}\Big(\nabla^2_z f(|x|,z)a_Z\big(f(|x|,z),z\big)\Big) - \partial_y \nabla_z f(|x|,z)^\top \sigma_Z\big(f(|x|,z),z\big)\sigma_X(x,z)\bigg).
\end{aligned} 
\end{equation}{ }

Next, we wish to write the dispersion terms as integrated against a Brownian Motion $W$, which satisfies $d[W,\overline W]_t = \rho(f(|X_t|,Z_t),Z_t)dt$.
 By inspecting the last two lines in \eqref{eq:dF_Z}, we see that this can be accomplished if the dispersion $\sigma_X$ in \eqref{eq:dX_Z} is chosen via 
\begin{equation} \label{eq:sigmaX}
\sigma_X(x,z) = \sign(x)\bigg(\frac{\rho(f(|x|,z),z)}{\sqrt{1-\|\rho(f(|x|,z),z)\|^2}} - \frac{ \sigma_Z(f(|x|,z),z)^\top \nabla_z f(|x|,z)}{\partial_y f(|x|,z)}\bigg), \qquad (x,z) \in \R \times E.
\end{equation}
With the choices \eqref{eq:alphaZ} and \eqref{eq:sigmaX}, the dynamics \eqref{eq:dF_Z} collapse to
\begin{equation} \label{eq:df_Z}
df(|X_t|,Z_t) = b\big(f(|X_t|,Z_t),Z_t\big)dt + \frac{\partial_y f(|X_t|,Z_t)}{\sqrt{1-\|\rho(f(|X_t|,Z_t),Z_t)\|^2}}dW_t + d\Phi_t,
\end{equation}
where 
\begin{equation} \label{eq:WZ}
    W =  \int_0^\cdot \sign(X_t)\sqrt{1-\big\|\rho\big(f(|X_t|,Z_t),Z_t\big)\big\|^2}\,dB_t + \int_0^\cdot \rho\big(f(|X_t|,Z_t),Z_t\big)^\top d\overline W_t 
\end{equation}
is scalar Brownian Motion, and 
   \begin{equation} \label{eq:Phi_Z_prelim} 
   \Phi = \int_0^\cdot \partial_y f(0,Z_t)dL_t^0(X).
   \end{equation}
As such, for $(Y,Z) = (f(|X|,Z),Z)$ to satisfy \eqref{eq:RSDE_YZ}, it just remains to match the dispersion coefficient in front of $W$ and ensure that $\Phi$ only accumulates on the set $\{t \geq 0: Y_t = 0\}$. This leads to the parameter-dependent initial value problem
\begin{equation}
\label{eq:ODE_Z}
\begin{cases}
    \partial_y f(y,z) = \widehat \sigma(f(y,z),z), & (y,z) \in [0,\infty) \times E, \\
    f(0,z) = 0, & z \in E
\end{cases}
\end{equation}
for the function $f:I \times E \to I$, where $\widehat \sigma$ is given by \eqref{eq:sigma_hat}.
For each fixed $z \in E$, the equation \eqref{eq:ODE_Z} is a nonlinear ODE in the variable $y$, and of the exact same form as \eqref{eq:nonlinear_ODE} with $\widehat \sigma (\cdot,z)$ in place of $\sigma(\cdot)$. As such, it admits the solution 
\begin{equation} \label{eq:f_Z}
    f(\cdot,z) = \kappa^{-1}(\cdot,z), \qquad \text{where} \qquad  \kappa(\xi,z) = \int_0^\xi \frac{1}{\widehat \sigma(u,z)}du \qquad \text{for all } z \in E, \,  \xi \in I.
\end{equation}

This establishes the following result.
\begin{theorem} \label{thm:halfline_Z}
    Let Assumption~\ref{ass:coefficients_Z} on the coefficients of the system \eqref{eq:RSDE_YZ} hold. Then:
    \begin{enumerate}[label = (\roman*),noitemsep]
        \item \label{item:f_halfline_Z} For every $z \in E$, the function $f(\cdot,z)$ of \eqref{eq:f_Z} is the unique solution to the parameter-dependent nonlinear ODE \eqref{eq:ODE_Z}, where $\widehat \sigma(\cdot,z)$ is given by \eqref{eq:sigma_hat}. Moreover, $f$ is of class $C^1$, $f(\cdot,z)$ belongs to the class $W^{2,1}_{\mathrm{loc}}(I)$ for every $z \in E$, and $f(x,\cdot)$ is of class $C^2$ for every $x \in I$;
        \item \label{item:SDE_halfline_Z} Consider the SDE system
        \begin{equation} \label{eq:SDE_system_halfline}
        \begin{aligned}
            dX_t & = \alpha(X_t,Z_t)dt + dB_t + \sigma_X(X_t,Z_t)^\top d\overline W_t, \\
            dZ_t & = \overline b_Z(X_t,Z_t)dt + \overline \sigma_Z(X_t,Z_t)d\overline W_t,
        \end{aligned}
        \end{equation}
        where $\overline b_Z(x,z) = b_Z(f(|x|,z),z)$,  $\overline \sigma_Z (x,z) = \sigma_Z(f(|x|,z),z)$, $B$ and $\overline W$ are independent univariate and $m$-dimensional Brownian Motions respectively, and $\alpha$, $\sigma_X$ are given by \eqref{eq:alphaZ} and \eqref{eq:sigmaX} respectively. 
        This system has a globally defined weak solution, which is unique in law, for any initial values $(X_0,Z_0) = (x,z) \in \R \times E$;
        \item \label{item:folding_halfline_Z} The process $(Y,Z) = (f(|X|,Z),Z)$, with $(X,Z)$ as in \eqref{eq:SDE_system_halfline} and $f$ given by \eqref{eq:f_Z}, satisfies the RSDE system \eqref{eq:RSDE_YZ} on $I \times E$ with initial condition $(Y_0,Z_0) = (f(|X_0|,Z_0),Z_0)$, scalar Brownian Motion $W$ as in \eqref{eq:WZ}, and reflection term
        \begin{equation} \label{eq:Phi_Z}
            \Phi = \int_0^\cdot \sigma(0,Z_t)\sqrt{1-\|\rho(0,Z_t)\|^2}\, dL_t^0(X).
        \end{equation}    \end{enumerate}
\end{theorem}
\begin{proof}
    First, we note that for every $z \in E$, the function $\widehat \sigma(\cdot,z)$ satisfies the conditions of Lemma~\ref{lem:ODE}. This yields the unique solution $f(\cdot,z)$ given by \eqref{eq:f_Z} and establishes that $f(\cdot,z) \in W^{2,1}_{\mathrm{loc}}(I)$.  Additionally, since $\kappa$ is jointly $C^1$ with $\partial_\xi \kappa = 1/\widehat \sigma > 0$, the Inverse Function Theorem ensures that $f$ is of class $C^1$. Moreover, since $\widehat \sigma(y,\cdot)$ is of class $C^2_b$, the function  $f(x,\cdot)$ is of class $C^2$ for every $x \in I$, proving \ref{item:f_halfline_Z}.   
     
     To establish well-posedness for \eqref{eq:SDE_system_halfline} we first note that $\partial_y f(y,z) = \widehat \sigma(f(y,z),z) \ge \epsilon\sqrt{1-\delta^2} > 0$ for all $(y,z) \in I \times E$. Hence, by boundedness of the coefficients assumed in Assumption~\ref{ass:coefficients_Z}, we see that $\alpha$ and $\overline b_Z$ are bounded. Next, we note that the joint diffusion matrix of $(X,Z)$ is given in block form by
     \[A(x,z) = \begin{bmatrix}
        1 + \|\sigma_X(x,z)\|^2 & \sigma_X(x,z)^\top \overline \sigma_Z(x,z)^\top   \\
        \overline \sigma_Z(x,z)\sigma_X(x,z) & \overline \sigma_Z(x,z)\overline\sigma_Z(x,z)^\top 
     \end{bmatrix}.
     \] The function $A$ is uniformly elliptic and  continuous outside the set $\{x = 0\}$, with the discontinuity arising due to the block off-diagonal terms containing a $\sign(x)$ term inherited from its presence in equation \eqref{eq:sigmaX} defining $\sigma_X$. However, since this discontinuity is of co-dimension one, the system \eqref{eq:SDE_system_halfline} still admits a global solution which is unique in law. Indeed, this follows from \cite[Remark~3.4]{krylov2004on}, applicable here because $A$ is uniformly continuous on $\{x > 0\} \times E$ and $\{x < 0\} \times E$; this follows from the Lipschitz continuity of  $\sigma_X$ and $\overline \sigma_Z$ on both of these domains. This claim \ref{item:SDE_halfline_Z} is established.

     To establish the folding representation, we note that the regularity of $f$ allows us to apply the It\^o\,--Tanaka formula to $ F(X,Z) = f(|X|,Z)$. This leads to the dynamics \eqref{eq:dF_Z}, and substituting  \eqref{eq:alphaZ} for $\alpha$ and \eqref{eq:sigmaX} for $\sigma_X$ leads to \eqref{eq:df_Z}, with the scalar Brownian Motion $W$ given by \eqref{eq:WZ} and the process of finite variation $\Phi$ given by \eqref{eq:Phi_Z_prelim}. Recalling that $f$ satisfies the parameter-dependent ODE \eqref{eq:ODE_Z}, we see that the diffusion coefficient in \eqref{eq:df_Z} is equal to $\sigma(f(|X_t|,Z_t),Z_t)$ and the reflection term has the representation \eqref{eq:Phi_Z}. This establishes that $(Y,Z) = (f(|X|,Z),Z)$ satisfies the system \eqref{eq:RSDE_YZ}, and completes the proof.
    \end{proof}

    It is clear that the methods developed here, together with those of Subsection~\ref{sec:unit_interval}, can also be applied to obtain a folding representation on $I \times E$, for $(Y,Z)$ as in \eqref{eq:RSDE_YZ}, when $I = [0,1]$ is a bounded interval. We state this result without proof, as it follows by combining the construction of Theorem~\ref{thm:halfline_Z} with the approach of Subsection~\ref{sec:unit_interval}. The main modification is a $z$-dependent normalization constant $\nu(z)$, which is as in Lemma~\ref{lem:ODE_nu} with $\widehat \sigma(\cdot,z)$ in place of $\sigma$. Note that the assumed regularity on $\widehat \sigma$ ensures that the resulting function $\nu$ is of class $C^2$.
\begin{theorem} \label{thm:interval_Z}
    Let Assumption~\ref{ass:coefficients_Z} on the coefficients of the system \eqref{eq:RSDE_YZ} hold for $I = [0,1]$.
    \begin{enumerate}[label = (\roman*),noitemsep]
        \item There exist unique $f:I \times E \to I$ with $f(\cdot,z)$ belonging to $W^{2,1}(I)$ for every $z \in E$, and $\nu: E \to (0,\infty)$ satisfying
        \[\begin{cases}
            \partial_y f(y,z)  = \frac{1}{\nu(z)}\widehat \sigma(f(y,z),z), &   (y,z) \in I \times E, \\
            f(0,z) = 0, \quad  f(1,z) = 1, & z \in E,
        \end{cases}\]
        where $\widehat \sigma$ is given by \eqref{eq:sigma_hat}.
        These functions are explicitly given by the expressions
        \begin{equation} 
        \label{eq:nu_Z} 
        f(y,z) = \kappa^{-1}\bigg(\frac{y}{\nu(z)},z\bigg), \qquad \text{where} \qquad \nu(z) = \bigg(\int_0^1 \frac{1}{\widehat \sigma(u,z)}du\bigg)^{-1}
        \end{equation}
        and $\kappa^{-1}(\cdot,z)$ the inverse of the function $\kappa(\cdot,z)$ in \eqref{eq:f_Z}. Moreover, $f(\cdot,\cdot)$ is of class $C^1$, $f(x,\cdot)$ is of class $C^2$ for every $x \in I$, and $\nu(\cdot)$ is of class $C^2$.
        \item Consider the SDE system 
        \begin{equation}
    \label{eq:SDE_system_interval}
        \begin{aligned}
            dX_t & = \alpha(X_t,Z_t)dt + \nu(Z_t)\big(dB_t + \sigma_X(\overline X_t,Z_t)^\top d\overline W_t\big), \\
            dZ_t & = \overline b_Z(X_t,Z_t)dt + \overline \sigma_Z(X_t,Z_t)d\overline W_t,
        \end{aligned}
        \end{equation}
        where 
        $\overline b_Z(x,z) = b_Z(f(|\overline x|,z),z), \overline \sigma_Z(x,z) = \sigma_Z(f(|\overline x|,z),z)$
        with $\overline x = x - \Even(x)$, where $f(\cdot,\cdot)$ and $\nu(\cdot)$ are given by \eqref{eq:nu_Z}, $\sigma_X(\cdot,\cdot)$ is given by \eqref{eq:sigmaX},
        and where
        \[\alpha(x,z) = \sign(\overline x) \,\zeta_{\nu(z)}(\overline x,z), \qquad (x,z) \in [0,1] \times E,\]
        with $\zeta$ given by \eqref{eq:zetaZ}. This system has a globally defined weak solution, which is unique in law, for any initial values $(X_0,Z_0) = (x,z) \in \R \times E$.
        \item The process $(Y,Z) = (f(|\overline X|,Z),Z)$, with $(X,Z)$ as in \eqref{eq:SDE_system_interval}, $f$ as in \eqref{eq:nu_Z} and $\overline X = X - \Even(X)$,  satisfies the RSDE system \eqref{eq:RSDE_YZ} on $I\times E$ with initial condition $(Y_0,Z_0) = (f(|\overline X_0|,Z_0),Z_0)$, scalar Brownian Motion  \begin{align*} 
    W & =  \int_0^\cdot \sign(\overline X_t)\sqrt{1-\big\|\rho\big(f(|\overline X_t|,Z_t),Z_t\big)\big\|^2}\,dB_t + \int_0^\cdot \rho\big(f(|\overline X_t|,Z_t),Z_t\big)^\top d\overline W_t, 
\intertext{and reflection term}
        \Phi & = \sum_{n \in \Z}\int_0^\cdot  \frac{1}{\nu(Z_t)} \big(\sigma(0,Z_t)\sqrt{1-\|\rho(0,Z_t)\|^2}\, dL_t^{2n}(X)\\
        & \qquad \qquad \qquad \qquad   - \sigma(1,Z_t)\sqrt{1-\|\rho(1,Z_t)\|^2}\, dL_t^{2n+1}(X)\big).
        \end{align*}
    \end{enumerate}
\end{theorem}

\section{Multidimensional domains: orthant and hypercube} \label{sec:multidimensional}

\subsection{The positive orthant} \label{sec:orthant}
Here we study the folding representation question in the positive orthant $D = [0,\infty)^d \subset \R^d$ with inward normal reflection on its boundary. Specifically, we start with an RSDE on $D$ given symbolically by \eqref{eq:RSDE}, but where $Y = (Y^1,\dots,Y^d)^\top $ is now a $d$-dimensional process, and the coefficients $b:D \to \R^d$
and $\sigma:D \to \R^{d \times d}$ are vector- and matrix-valued, respectively. The reflection term $\Phi = (\Phi^1,\dots,\Phi^d)^\top $ is now a multivariate process with continuous paths of finite first variation on compact intervals, required to accumulate only on the boundary, i.e., satisfying
\[\mathrm{supp}(d\Phi_\cdot) \subset \{t \geq 0: Y_t \in \partial D\},\] and with inward reflection in the normal direction, i.e., 
\begin{equation} \label{eq:bivariate_Phi}
\Phi_t = \int_0^t n(Y_s)\, d\TV{\Phi}_s = \sum_{i=1}^d e_i \int_0^t 1_{\{Y^i_s = 0\}} d\TV{\Phi}_s; \qquad t \geq 0.
\end{equation}
 Here $n(y)$ is the inward-pointing normal vector at $y \in \partial D$, $e_1,\dots,e_d$ are the standard basis vectors, and $\TV{\Phi}_t$ is the total variation of $\Phi$ on $[0,t]$; that is,
 \begin{equation} \label{eq:TV}
     \TV{\Phi}_t = \sup_\pi  \sum_{k=1}^n \|\Phi_{t_k} - \Phi_{t_{k-1}}\|,
 \end{equation}
 with the supremum taken over all partitions $\pi$ of the form $0 = t_0 < t_1 < \dots < t_n=t$ for some $n \in \mathbb{N}$.

We look now for a vector-valued folding function $F = (F^1,\dots, F^d)^\top : \R^d \to D$ and a standard diffusion $X = (X^1,\dots,X^d)^\top$ of the form $dX^j_t = \alpha^j(X_t)dt + dB^j_t$, $j=1,\dots,d$ as in \eqref{eq:X_SDE}, but now with independent standard Brownian Motions $B^1,\dots,B^d$ and with suitable drift functions $\alpha^j: \R^d \to \R$, for which the relationship \eqref{eq:goal} holds.

Motivated by our analysis in the one-dimensional case, a natural Ansatz is to take $F^i(x) = f^i(|x|)$  for $i=1,\dots,d$, where $f^i:[0,\infty) \to [0,\infty)$ is a sufficiently regular function to be determined, and we write $|x|$ for the vector $(|x^1|,\dots,|x^d|)$.

Now applying the It\^o\,--Tanaka formula to each $F^i(X)$ gives
\begin{equation} \label{eq:dF_quadrant}
\begin{split} 
dF^i(X_t) & = \sum_{j=1}^d \partial_j f^i(|X_t|)d|X^j|_t + \tfrac{1}{2}\sum_{j=1}^d \partial_{jj} f^i(|X_t|)d[|X^j|]_t    \\
& = \sum_{j=1}^d \big( \partial_jf^i(|X_t|)\sign(X_t^j)\alpha^j(X_t) + \tfrac{1}{2}\partial_{jj}f^i(|X_t|) \big)dt + \sum_{j=1}^d \partial_j f^i(|X_t|)\sign(X_t^j)dB^j_t  \\
& \qquad + \sum_{j=1}^d \partial_j f^i\big(|X_t| \odot ({\bf 1}_d - e_j)\big)dL_t^0(X^j), 
\end{split}
\end{equation}
where ${\bf 1}_d$ denotes the $d$-dimensional vector of ones and $\odot$ represents componentwise product.
Next, we write $J_f$ for the Jacobian matrix of the function $f$, and define the standard $d$-dimensional Brownian Motion $W = \int_0^\cdot \mathrm{sign}(X_t)dB_t$, where $\sign(x)$ denotes the diagonal matrix with entries $\sign(x)^{ii} = \sign(x^i)$. We can write the differential of the process $F(X)$ above in matrix-vector form as
\begin{equation} \label{eq:DF_quadrant_vector}
dF(X_t) = \big(J_f(|X_t|)\sign(X_t) \alpha(X_t) + \tfrac{1}{2} \Delta f(|X_t|)\big)dt + J_f(|X_t|)dW_t +  J_f(|X_t|)dL_t^0(X),
\end{equation}
where we set $\Delta f = (\Delta f^1, \dots, \Delta f^d)^\top$, with $\Delta$ denoting the Laplacian operator, and $L^0(X) = (L^0(X^1),\dots,L^0(X^d))^\top $.
Comparing term by term with the RSDE \eqref{eq:RSDE} leads us to the conditions
\begin{align}
    J_f(y) & = \sigma\big(f(y)\big), & \text{for } y \in D, \label{eq:nonlinear_PDE_quadrant}\\
     b\big(f(|x|)\big) & = J_f(|x|)\sign(x) \alpha(x) + \tfrac{1}{2}\Delta f(|x|), & \text{for } x \in \R^d, \label{eq:drift_equation_quadrant} \\
    d\Phi^i_t & = \sum_{j=1}^d (J_f)^{ij}\big(|X_t|\odot ({\bf 1}_d-e_j)\big)dL^0_t(X^j) & \text{for } i =1,\dots,d,\quad t \geq 0.\label{eq:Phi_equation_quadrant}
\end{align}
In a manner analogous to the one-dimensional case, if we can find a function $f = (f^1,\dots,f^d)^\top$ with components $f^i:[0,\infty) \to [0,\infty)$, $i=1,\dots,d$ satisfying the first-order PDE \eqref{eq:nonlinear_PDE_quadrant}, then \eqref{eq:drift_equation_quadrant} can be satisfied by taking
\begin{equation}\label{eq:alpha_implicit_quadrant} 
\alpha(x) = \sign(x)J_f^{-1}(|x|)\Big(b\big(f(|x|)\big) - \tfrac{1}{2}\Delta f(|x|)\Big).
\end{equation}

 Hence, we focus on the first-order nonlinear PDE \eqref{eq:nonlinear_PDE_quadrant}. To uncover what boundary conditions should be enforced, we inspect \eqref{eq:Phi_equation_quadrant}. Since $Y = f(|X|)$, and the local times of the components of $X$ at zero appear in the decomposition of $f(|X|)$, it is clear that for the reflecting term $\Phi$ to accumulate only on the boundary, we need to have $X^i_t = 0 \implies f^i(|X_t|) = 0$ for every $i=1,\dots,d$; and this leads to the condition 
\begin{equation} \label{eq:f_boundary_quadrant}
     f^i\big(y \odot ({\bf 1}_d-e_i)\big) = 0, \qquad  \text{for } i=1,\dots,d, \quad  y \in D.
\end{equation}
We will see below that an appropriate condition on $\sigma$ at the boundary, namely \eqref{eq:sigma_offdiagonal_boundary_condition}, ensures that all of the local-time terms in \eqref{eq:Phi_equation_quadrant} vanish except the one corresponding to $X^i$, and that the surviving term renders $\Phi$ inward-pointing; we make this precise after determining the solvability conditions for \eqref{eq:nonlinear_PDE_quadrant}. 

We now turn our attention to solving the partial differential equation (PDE) \eqref{eq:nonlinear_PDE_quadrant}. To this end, we assume that the dispersion coefficient $\sigma$ is globally Lipschitz continuous and that the symmetric part of $\sigma(y)$ is positive definite for every $y \in D$; that is,
\begin{equation} \label{eq:pd}
v^\top \big(\sigma(y) + \sigma(y)^\top\big) v > 0 \qquad \text{for every } v \in \R^d\setminus \{0\}.
\end{equation} In particular, this implies that $\sigma(y)$ is an invertible matrix for every $y \in D$. 
In general, the system \eqref{eq:nonlinear_PDE_quadrant} is overdetermined, but we will be able to find necessary and sufficient conditions on $\sigma$ that ensure a solution exists. To uncover these conditions, we first decompose the PDE \eqref{eq:nonlinear_PDE_quadrant} into the system of ordinary differential equations (ODEs), 
\begin{equation} \label{eq:nonlinear_ODE_quadrant}
    \partial_j f(y) = \begin{pmatrix}
    \partial_j f^1(y) \\ \vdots \\ \partial_j f^d(y)
\end{pmatrix} = \begin{pmatrix}
    \sigma^{1j}(f^1(y), \dots , f^d(y)) \\ \vdots \\ \sigma^{d\,\!j}(f^1(y),\dots, f^d(y))
\end{pmatrix} = \sigma^{\cdot j}\big(f(y)\big), \qquad j = 1,\dots,d,
\end{equation}
where $\sigma^{\cdot j}$ denotes the $j$th column of $\sigma$. For each $j$, this is a coupled system of nonlinear ODEs with $y^j$ as the variable. We need to assume that
\begin{equation} \label{eq:sigma_boundary_conditions} 
\sigma^{ij}\big(y\odot({\bf 1}_d-e_i)\big) \geq 0, \qquad \forall\ y \in D, \quad i,j=1,\dots,d,
\end{equation}
which ensures that the columns of $\sigma$ at the boundary are pointing inward. Indeed, if \eqref{eq:sigma_boundary_conditions} fails, the flow defining $f$ would exit the orthant, so that $f$ could not possibly take values in the orthant and be a global solution (see \cite[Theorem~8.5.11]{hormander2003the}).

To make progress, we first consider this problem with only the initial condition $f(0) = 0$ imposed, rather than the more restrictive condition \eqref{eq:f_boundary_quadrant}. Under the global Lipschitz condition on $\sigma$ and the condition \eqref{eq:sigma_boundary_conditions}, we obtain flow maps $\eta^j_t(z)$ induced by the ODEs \eqref{eq:nonlinear_ODE_quadrant}, which satisfy for all $t \geq 0$ and $z \in D$,
\begin{equation} \label{eqn:flow_ODE}
    \frac{d}{dt}\eta^j_t(z) = \sigma^{\cdot j}\big(\eta^j_t(z)\big); \quad \eta^j_0(z) = z, \qquad j=1,\dots,d.
\end{equation}
Now, 
initiating at $z = 0$ and flowing along the $y^1$-axis, any solution $f$ of \eqref{eq:nonlinear_PDE_quadrant} with $f(0) = 0$ must satisfy $f(y^1e_1) = \eta^1_{y^1}(0)$. Next, flowing in the $y^2$ direction starting from $z = \eta^1_{y^1}(0)$ gives $f(y^1e_1+y^2e_2) = \eta^2_{y^2}(\eta^1_{y^1}(0))$. Proceeding in this way, we obtain the representation  $f(y) = (\eta^d_{y^d} \circ \cdots \circ \eta^1_{y^1})(0)$. Of course, one can permute the order in which the flows are applied, and arrive at the representation  
\begin{equation} \label{eq:f_quadrant}
f(y) = \Big(\eta^{p(d)}_{y^{p(d)}} \circ \cdots \circ \eta^{p(1)}_{y^{p(1)}}\Big)(0), \quad y \in D,
\end{equation} 
for any permutation $p(\cdot)$ of $\{1,\dots,d\}$.
 It follows that a necessary condition for our PDE \eqref{eq:nonlinear_PDE_quadrant} to have a solution, is for the flow maps in \eqref{eqn:flow_ODE} to commute; that is, we need to have \[\eta^i_{y^i} \circ \eta^j_{y^j} = \eta^j_{y^j} \circ \eta^i_{y^i} \qquad \text{for every } i,j=1,\dots,d.\]

Standard computations,  which compare second cross derivatives taken in different order, yield that these flows commute if, and only if, the \emph{Frobenius condition}
\[\sum_{\ell=1}^d \big(\sigma^{\ell i}(y)\partial_\ell \sigma^{kj}(y) - \sigma^{\ell j}(y)\partial_\ell\sigma^{ki}(y)\big) = 0 \qquad \text{for  a.e.\ } y \in D \quad  \text{and all} \quad i,j,k=1,\dots,d  \]
holds. More compactly, this condition can be written as
\begin{equation} \label{eq:Lie_bracket} [\sigma^{\cdot i},\sigma^{\cdot j}]=0, \qquad i,j=1,\dots,d,
\end{equation}
where $[\cdot,\cdot]$ denotes the Lie bracket of two vector fields and the equality in \eqref{eq:Lie_bracket} is understood to hold almost everywhere. In this case, the function $f$ of \eqref{eq:f_quadrant} is well-defined and satisfies the nonlinear, first-order PDE \eqref{eq:nonlinear_PDE_quadrant} together with the initial condition $f(0) = 0$. 

We now obtain conditions on $\sigma$, so that \eqref{eq:f_boundary_quadrant} holds. By differentiating $f^i$ in the $j$th coordinate for $j \ne i$, we see from \eqref{eq:f_boundary_quadrant} and \eqref{eq:nonlinear_PDE_quadrant} that we must have
\[ 0  = \partial_jf^i\big(y\odot ({\bf 1}_d-e_i)\big) = \sigma^{ij}\Big(f\big(y\odot ({\bf 1}_d-e_i)\big)\Big), \qquad \forall\ y \in D, \quad i \ne j.\]
In Lemma~\ref{lem:ODE_quadrant} below, we will show that the range of $f$ is all of $D$, which leads to the compatibility requirement
\begin{equation} \label{eq:sigma_offdiagonal_boundary_condition}
    \sigma^{ij}\big(y \odot ({\bf 1}_d - e_i)\big) = 0, \qquad \forall\ y \in D, \quad i \ne j.
\end{equation}
This is a substantial strengthening of the condition \eqref{eq:sigma_boundary_conditions} for the off-diagonal entries of $\sigma$. Note, however,  that the diagonal terms satisfy 
\[\sigma^{ii}\big(y\odot({\bf 1}_d-e_i)\big) > 0 \qquad \text{for all }y \in D \text{ and } i=1,\dots,d,\] since the symmetric part of $\sigma(y)$ is positive definite for all $y \in D$, as required by \eqref{eq:pd}.
 
We aggregate the conclusions of this discussion in the following result, whose proof we defer to Appendix~\ref{sec:orthant_proof}. \begin{lem} \label{lem:ODE_quadrant}
    Let $\sigma:D \to \R^{d \times d}$ be a globally Lipschitz continuous function, whose symmetric part is positive definite. Then a solution $f:D \to D$ to the nonlinear, first-order PDE \eqref{eq:nonlinear_PDE_quadrant} satisfying \eqref{eq:f_boundary_quadrant} exists if, and only if, $\sigma$ satisfies the Frobenius condition \eqref{eq:Lie_bracket} and the compatibility requirement \eqref{eq:sigma_offdiagonal_boundary_condition}. 
    In this case the solution $f$ is unique, belongs to $W^{2,\infty}_{\mathrm{loc}}(D)$, has all of $D$ as its range, and admits the representation  \eqref{eq:f_quadrant}. 
\end{lem}

Returning to \eqref{eq:Phi_equation_quadrant}, the local time term indexed by $j$ has integrand given by the expression
$(J_f)^{ij}(|X_t|\odot({\bf 1}_d - e_j)) = \sigma^{ij}(f(|X_t| \odot ({\bf 1}_d - e_j)))$, which vanishes on the set $\{y^j = 0\}$ for $i \ne j$ by \eqref{eq:sigma_offdiagonal_boundary_condition}.
With these preparations in hand, we are ready to establish a folding representation for \eqref{eq:RSDE} on the positive orthant with normal reflection.
\begin{theorem} \label{thm:main_quadrant}
 Consider measurable functions $b:D \to \R^d$ and $\sigma:D \to \R^{d \times d}$, such that $\sigma$ satisfies the assumptions of Lemma~\ref{lem:ODE_quadrant}. 
Further, suppose that the function 
\begin{equation} \label{eq:xi}
\xi(y) = \sigma^{-1}(y)\bigg(b(y) - \tfrac{1}{2}\sum_{j=1}^d (\sigma^{\cdot j}\cdot\nabla)\sigma^{\cdot j}(y)\bigg), \qquad y \in D,
\end{equation} is bounded, where
$ (\sum_{j=1}^d (\sigma^{\cdot j}\cdot\nabla)\sigma^{\cdot j})^i = \sum_{j,\ell=1}^d \sigma^{\ell j}\,\partial_\ell\sigma^{ij}$ for $i=1,\dots,d$.
Then,
\begin{enumerate}[label = (\roman*),noitemsep]
    \item \label{item:f_orthant} the function $f:D \to D$ of \eqref{eq:f_quadrant} is well-defined, of class $W^{2,\infty}_{\mathrm{loc}}(D)$, has $D$ as its range, and is the unique solution of the PDE \eqref{eq:nonlinear_PDE_quadrant} satisfying \eqref{eq:f_boundary_quadrant} on the boundary;
    \item with $\xi,f$ as above and
    \begin{equation} \label{eq:alpha_quadrant}
\alpha(x) = \sign(x)\, \xi\big(f(|x|)\big), \qquad \text{for a.e. } x \in \R^d,
    \end{equation}
 the SDE \eqref{eq:X_SDE} has a pathwise unique, strong solution for every initial value $X_0 \in \R^d$;
    \item the process $Y = f(|X|)$ satisfies the RSDE \eqref{eq:RSDE} on $D$ with initial condition $Y_0 = f(|X_0|)$, Brownian Motion $W = \int_0^\cdot\sign(X_t)dB_t$, and normally reflecting processes
    \begin{equation}  \label{eq:Phi_representation_quadrant}
    \Phi^i = \int_0^\cdot \sigma^{ii}\Big(f\big(|X_t|\odot ({\bf 1}_d - e_i)\big)\Big)dL_t^0(X^i), \qquad i=1,\dots,d.
    \end{equation}  
    \end{enumerate}
\end{theorem}
\begin{proof}
    Item \ref{item:f_orthant} is simply a restatement of Lemma~\ref{lem:ODE_quadrant}. The expression for $\alpha$ in \eqref{eq:alpha_quadrant} is precisely the expression \eqref{eq:alpha_implicit_quadrant}, which we simplified by using the facts $J_f^{-1}(|x|) = \sigma^{-1}(f(|x|))$, courtesy of \eqref{eq:nonlinear_PDE_quadrant}, and $\Delta  f^i = \sum_{j,\ell=1}^d \sigma^{\ell j} (f)\partial_\ell \sigma^{i j}(f) = (\sum_{j=1}^d ((\sigma^{\cdot j} \cdot \nabla)\sigma^{\cdot j}))^i(f)$, obtained by differentiating \eqref{eq:nonlinear_PDE_quadrant}. By assumption, $\alpha$ is bounded so \eqref{eq:X_SDE} has a pathwise unique, strong solution for every initial condition $X_0 \in \R^d$ (see \cite{veretennikov1981on}).
    Now, the computation \eqref{eq:dF_quadrant} shows that $F(X) = f(|X|)$ satisfies \eqref{eq:DF_quadrant_vector}. This establishes that the diffusion process $Y = f(|X|)$ has $b(\cdot)$ and $\sigma(\cdot)$ as its drift and dispersion coefficients, courtesy of \eqref{eq:nonlinear_PDE_quadrant} and \eqref{eq:drift_equation_quadrant} being satisfied. 
    
    The expression for the reflection term \eqref{eq:Phi_representation_quadrant} follows because $J_f(y) = \sigma (f(y))$ and the boundary condition \eqref{eq:sigma_offdiagonal_boundary_condition} on $\sigma$ holds; the latter ensures the off-diagonal terms appearing in \eqref{eq:Phi_equation_quadrant} vanish. Note that, since $X^i_t = 0 \implies Y^i_t = 0$, courtesy of \eqref{eq:f_boundary_quadrant}, we have $\mathrm{supp}(d\Phi_\cdot) \subset \{t\geq 0: Y_t \in \partial D\}$. Moreover,  the positivity of the diagonal entries of $\sigma$ ensures that the reflecting term $\Phi$ is inward-pointing; that is, \eqref{eq:bivariate_Phi} holds. This establishes that $Y = f(|X|)$ solves \eqref{eq:RSDE} on $D$ with normal reflection, and completes the proof. 
    \end{proof}

    \subsection{The unit hypercube} \label{sec:hypercube} In this subsection, we study the folding representation problem on the domain $D = [0,1]^d$. That is, we seek again a folding representation of the form \eqref{eq:goal} for the process $Y$ satisfying the RSDE \eqref{eq:RSDE} and with a suitable $d$-dimensional diffusion $X$, when the domain $D$ is the unit hypercube. As with the orthant, we study this problem with normal reflection, which requires the reflection process $\Phi$ to be of finite variation, carried on the set $\{t \geq 0: Y_t \in \partial D\}$, and to satisfy
    \begin{equation} \label{eq:cube_Phi} 
    \Phi_t = \sum_{i=1}^d \int_0^t 1_{\{Y^i_s = 0\}}e_i d\TV{\Phi}_s -  \sum_{i=1}^d \int_0^t 1_{\{Y^i_s =1 \}}e_i d\TV{\Phi}_s, \qquad t \ge 0. 
    \end{equation}
    As in the case of the unit interval in one dimension, we will need additional degrees of freedom to ensure that the folding representation we construct has the correct boundary behavior on the newly introduced faces $\{y \in D: y^i=1\}$ for $i=1,\dots,d$. For this reason we consider $\R^d$-valued diffusions $X$ satisfying SDEs of the form \eqref{eq:X_SDE_nu}, where the drift coefficient $\alpha:\R^d \to \R^d$ and the constant diagonal dispersion matrix $\nu$ with positive diagonal entries $\nu^{ii} > 0$ are to be determined.

    Taking inspiration from the unit interval case of Subsection~\ref{sec:unit_interval}, we consider a folding function of the form 
    \[F(x) = \big(f^1\big(|\overline x|),\dots,f^d(|\overline x|)\big)^\top,\]
    where $\overline x = x- \Even(x) =  (x^1-\Even(x^1),\dots,x^d-\Even(x^d))^\top$. Using the It\^o\,--Tanaka formula on each component $F^i(X)$ of $F(X)$ leads, in a way similar to the computations of \eqref{eq:Ito_tanaka_interval} and \eqref{eq:dF_quadrant}, to the dynamics
    \begin{equation} \label{eq:Ito_tanaka_cube}
    \begin{aligned}
        dF(X_t) & = \big(J_f(|\overline X_t|)\sign(\overline X_t) \alpha(X_t) + \tfrac{1}{2}\Delta_\nu f(|\overline X_t|)\big) dt + J_f(|\overline X_t|)\nu dW_t \\
        & \qquad  + \sum_{n \in \Z} J_f(|\overline X_t|)dL^{2n}_t(X) -\sum_{n \in \Z} J_f(|\overline X_t|)dL^{2n+1}_t(X).
    \end{aligned}
    \end{equation}
    Here, we used again the notation $J_f$ for the Jacobian matrix of the function $f$, introduced the notation $\Delta_\nu f = (\Delta_\nu f^1,\dots,\Delta_\nu f^d)^\top$ for $\Delta_\nu f^i = \sum_{j=1}^d (\nu^{jj})^2 \partial_{jj} f^i$, defined the $d$-dimensional Brownian Motion $W = \int_0^\cdot \sign(\overline X_t)dB_t$, and, for any $a \in \R$, set $L_t^a(X) = (L_t^a(X^1),\dots,L_t^a(X^d))^\top$ to be a vector of local times.
    Comparing coefficients with the RSDE \eqref{eq:RSDE} leads to the PDE
    \begin{equation} \label{eq:nonlinear_PDE_cube}
    J_f(y) = \sigma\big(f(y)\big)\nu^{-1}, \qquad y \in D. 
    \end{equation}
    As in the case of the orthant, we impose the Lie bracket condition \eqref{eq:Lie_bracket} on $\sigma$, because this is a necessary condition for a solution to \eqref{eq:nonlinear_PDE_cube} to exist (note that $[\sigma^{\cdot i}, \sigma^{\cdot j}] = 0 \iff [(\sigma \nu^{-1})^{\cdot i}, (\sigma \nu^{-1})^{\cdot j}] = 0$ for any constant diagonal matrix $\nu$ with positive coefficients). By inspecting the local time terms in \eqref{eq:Ito_tanaka_cube}, it is clear that we need to enforce the boundary conditions
    \begin{equation} \label{eq:f_boundary_cube}
    f^i\big(y \odot ({\bf 1}_d - e_i)\big) = 0 \quad \text{and} \quad  f^i\big(y \odot ({\bf 1}_d - e_i)+e_i\big) = 1, \qquad \forall\ y \in D,\quad i=1,\dots,d, 
    \end{equation} in order to ensure that the reflecting term is supported on the boundary of the domain. Indeed, the first (resp., second) condition in \eqref{eq:f_boundary_cube} stipulates, that $f^i(y)$ equals zero (resp., one) whenever $y^i$ equals zero (resp., one), which is precisely the value that $y^i = x^i -\Even(x^i)$ takes when $x^i$ is an even (resp., odd) integer. Arguing by analogy with the orthant case in Subsection~\ref{sec:orthant}, we impose the requirement 
    \begin{equation} \label{eq:sigma_offdiagonal_boundary_condition_cube}
         \sigma^{ij}\big(y \odot ({\bf 1}_d-e_i)\big) = \sigma^{ij}\big(y \odot ({\bf 1}_d-e_i)+e_i\big) =0, \qquad \forall\ y\in D, \quad i \ne j 
    \end{equation} 
    on $\sigma$, which is the analogue of \eqref{eq:sigma_offdiagonal_boundary_condition} in the present setting. The requirement \eqref{eq:sigma_offdiagonal_boundary_condition_cube} will ensure that the solution to \eqref{eq:nonlinear_PDE_cube} we construct satisfies the boundary condition \eqref{eq:f_boundary_cube} on the faces of the cube.
    
    Now, proceeding as in the case of the orthant, given any fixed values $\nu^{ii} > 0$ for $i=1,\dots,d$ we can obtain a solution to \eqref{eq:nonlinear_PDE_cube}, which satisfies the first condition in \eqref{eq:f_boundary_cube}; that is, $f^i$ vanishes on the faces of the cube lying in the coordinate hyperplanes $\{y^i = 0\}$. Akin to the analysis of Subsection~\ref{sec:unit_interval} for the unit interval, we now seek values of $\nu$ that ensure the remaining boundary conditions in \eqref{eq:f_boundary_cube} are met. To this end, we note that we just have to guarantee that $f^i(e_i) = 1$  holds, since 
    \[\partial_j f^i(y \odot ({\bf 1}_d-e_i)+e_i) = \sigma^{ij}\big(f(y \odot ({\bf 1}_d-e_i)+e_i)\big) = 0, \qquad \forall\, y \in D, \quad i \ne j,
     \]
    courtesy of \eqref{eq:sigma_offdiagonal_boundary_condition_cube}, which ensures that $f^i$ is constant on faces of the cube where $y^i = 1$.
    From the PDE \eqref{eq:nonlinear_PDE_cube},  we see that
    \[ \partial_if^i(y^ie_i) = \frac{1}{\nu^{ii}} \sigma^{ii}\big(f(y^ie_i)\big) = \frac{1}{\nu^{ii}} \sigma^{ii}\big(f^i(y^ie_i)e_i\big), \qquad \forall\ y^i \in [0,1], \quad i=1,\dots,d, \]
    where in the final equality we used the fact that $f^j(y) = 0$  if $y^j = 0$.
    Proceeding as in the proof of Lemma~\ref{lem:ODE_nu}, we see that 
    \[\frac{1}{\nu^{ii}} = \int_0^1 \frac{\partial_i f^i(ue_i)}{\sigma^{ii}\big(f^i(ue_i)e_i\big)}du = \int_0^{f^i(e_i)}\frac{1}{\sigma^{ii}(ue_i)}du\]
    must hold, from which we conclude that $f^i(e_i) = 1 \iff \nu^{ii}= (\int_0^1\frac{1}{\sigma^{ii}(ue_i)}du)^{-1}$. 
    
    We have established the following result.
    \begin{lem}
 \label{lem:ODE_cube}
    Let $\sigma:D \to \R^{d \times d}$ be a Lipschitz continuous function, whose symmetric part is positive definite in the manner of \eqref{eq:pd}. Then a  solution $f:D \to D$ of the PDE \eqref{eq:nonlinear_PDE_cube}  satisfying the boundary condition \eqref{eq:f_boundary_cube} exists if, and only if, $\sigma$ satisfies \eqref{eq:Lie_bracket} and \eqref{eq:sigma_offdiagonal_boundary_condition_cube}, and the constant diagonal matrix $\nu$ appearing in \eqref{eq:nonlinear_PDE_cube} has positive entries
    \begin{equation} \label{eq:nus_cube}
    \nu^{ii} = \bigg(\int_0^1\frac{1}{\sigma^{ii}(ue_i)}du\bigg)^{-1}, \qquad \text{for } i=1,\dots,d.
    \end{equation}
    In this case, the solution $f$ is unique, belongs to $W^{2,\infty}(D)$, has all of $D$ as its range, and is given by the expression  \eqref{eq:f_quadrant}, where the flow maps $\eta_t^j$ satisfy the system of ODEs
    \begin{equation} \label{eq:flows_cube}
        \frac{d}{dt}\eta^j_t(z) = \frac{1}{\nu^{jj}}\sigma^{\cdot j}\big(\eta^j_t(z)\big); \quad \eta^j_0(z) = z, \qquad j=1,\dots,d,
    \end{equation}
for each $z \in D$ and all $t \geq 0$ for which $\eta^j_t(z) \in D$.
\end{lem}

We are now ready to establish a folding representation on the hypercube.
\begin{theorem} \label{thm:main_cube}
    Consider measurable functions $b:D \to \R^d$ and $\sigma:D \to \R^{d \times d}$, such that $\sigma$ satisfies the assumptions of Lemma~\ref{lem:ODE_cube}. Furthermore, let constants $\nu^{ii}$ be given by \eqref{eq:nus_cube}, construct the diagonal matrix $\nu$ of these elements, and suppose that the function $\xi$ defined in \eqref{eq:xi} is bounded. Then,
\begin{enumerate}[label = (\roman*),noitemsep]
    \item \label{item:f_cube} the function $f:D \to D$ given by equation \eqref{eq:f_quadrant} with flow maps satisfying \eqref{eq:flows_cube} is well-defined, belongs to $W^{2,\infty}(D)$, has $D$ as its range, and is the unique solution to the PDE \eqref{eq:nonlinear_PDE_cube} satisfying the boundary condition \eqref{eq:f_boundary_cube};
    \item with 
    \begin{equation} \label{eq:alpha_cube}
    \alpha(x) = \nu\sign(\overline x)\, \xi\big(f(|\overline x|)\big), \qquad \text{for a.e. } x \in \R^d,
    \end{equation}
    and $\overline x = x - \Even(x)$, 
     the SDE \eqref{eq:X_SDE_nu} has a pathwise unique, strong solution $X$ for every initial value $X_0 \in \R^d$;
    \item the process $Y = f(|X-\Even(X)|)$ satisfies the RSDE \eqref{eq:RSDE} on $D$ with initial condition $Y_0 = f(|X_0-\Even(X_0)|)$, normally reflecting processes
    \begin{equation}
\label{eq:Phi_representation_cube}
\begin{aligned} 
\Phi_t^i  = &  \frac{1}{\nu^{ii}}\sum_{n \in \Z} \int_0^t \sigma^{ii}\Big(f\big(|X_s - \Even(X_s)|\odot ({\bf 1}_d - e_i)\big)\Big)dL_s^{2n}(X^i) \\
&  - \frac{1}{\nu^{ii}}\sum_{n \in \Z} \int_0^t \sigma^{ii}\Big(f\big(|X_s - \Even(X_s)|\odot ({\bf 1}_d - e_i)+e_i\big)\Big)dL_s^{2n+1}(X^i), 
    \end{aligned}
\end{equation}
for $i=1,\dots,d$,  $ t \geq 0$, 
    and Brownian Motion $W = \int_0^\cdot\sign(X_t-\Even(X_t))dB_t$. 
    \end{enumerate}
\end{theorem}
\begin{proof}
Item \ref{item:f_cube} is simply a restatement of Lemma~\ref{lem:ODE_cube}. The expression for $\alpha$ in \eqref{eq:alpha_cube} ensures that the drift term in \eqref{eq:Ito_tanaka_cube} is equal to $b(f(|\overline x|))$. This follows from the fact that  $J_f^{-1}(y) = \nu \sigma^{-1}(y)$, courtesy of \eqref{eq:nonlinear_PDE_cube}, and that 
\[\Delta_\nu f^i = \sum_{j=1}^d (\nu^{jj})^2 \partial_{jj}f^i = \sum_{j=1}^d \nu^{jj}\partial_j \big(\sigma^{ij}(f)\big) = \sum_{j,\ell=1}^d \partial_\ell \sigma^{ij}(f)\sigma^{\ell j}(f) =     \bigg(\sum_{j=1}^d \big((\sigma^{\cdot j} \cdot \nabla)\sigma^{\cdot j}\big)^i\bigg)(f),\]
where we used \eqref{eq:nonlinear_PDE_cube} twice to replace derivatives of $f$ by terms involving $\nu$ and $\sigma$. By assumption, $\alpha$ is bounded, so \eqref{eq:X_SDE_nu} has a pathwise unique, strong solution for every initial condition $X_0 \in \R^d$ (as before, see \cite{veretennikov1981on}). The representation \eqref{eq:Ito_tanaka_cube} shows that the stochastic dynamics for $Y = f(|X-\Even(X)|)$ has $b(\cdot)$ and $\sigma(\cdot)$ as its drift and dispersion coefficients. 
    
    The expression for the reflection term \eqref{eq:Phi_representation_cube} follows from the fact that $J_f(y) = \sigma (f(y))\nu^{-1}$ and from the boundary condition \eqref{eq:sigma_offdiagonal_boundary_condition_cube}, which ensures the off-diagonal terms of $J_f$ vanish in the integral against the local time terms of \eqref{eq:Ito_tanaka_cube}. Note that, since $X^i_t = 2n \implies Y^i_t = 0$ and $X^i_t = 2n+1 \implies Y^i_t = 1$ for any $n \in \Z$, the process $\Phi$ is carried by the set $\{t\geq 0: Y_t \in \partial D\}$. Moreover,  the positivity of the diagonal entries of $\sigma$ ensures that the reflecting term $\Phi$ is inward-pointing; that is, \eqref{eq:cube_Phi} holds. This establishes that $Y = f(|X-\Even(X)|)$ solves \eqref{eq:RSDE} on $D$ with normal reflection, and completes the proof.    
\end{proof}

\section{Convex domains} \label{sec:convex_domain}
In this section we consider more general closed and bounded convex domains $D \subset\R^d$ for $d \ge 2$. Here, the situation is more delicate than in the previously considered domains, as the geometry of the domain and the direction of reflection play a critical role. Nevertheless, we will be able to obtain a folding representation \eqref{eq:goal}, even when the reflection is oblique.
Concretely, we impose the following assumption regarding the domain $D$.

\begin{assum} \label{ass:domain}
 The domain $D \subset \R^d$ for $d \ge 2$ is a closed and bounded convex set containing the origin in its interior and has a boundary of class $C^3_+$; that is, the boundary $\partial D$ is of class $C^3$ and has everywhere positive curvature.
\end{assum} 
Under Assumption~\ref{ass:domain}, the domain $D$ has the representation
\begin{equation} \label{eq:gauge}
D = \{y \in \R^d: \rho_D(y) \leq 1\}, \qquad \text{where } \rho_D(y) := \inf\{r > 0: y \in rD\}
\end{equation}
 is the gauge function associated with $D$. In particular, the boundary of the domain has the representation
\[\partial D = \{y \in \R^d: \rho_D(y) = 1\}.\]
The function $\rho_D$ belongs to the class $C^3(\R^d\setminus \{0\})$; it inherits this regularity from the $C^3$ boundary of $D$, 
and satisfies 
\[v^\top \nabla^2 \rho_D(y)\, v > 0 \qquad \text{ for every } y \neq 0 \text{ and every } v \in y^\perp \setminus \{0\},\] due to the positive curvature of the boundary (see  \cite[Lemma~1.7.13]{schneider1993convex} and \cite[Section~2.5]{schneider1993convex}; in particular, \cite[Corollary~2.5.2]{schneider1993convex}). 

On such a domain $D$ we will study the RSDE \eqref{eq:RSDE} with oblique reflection. That is, in addition to the coefficients $b: D \to \R^d$ and $\sigma:D \to \R^{d \times d}$, we take a third coefficient $\gamma:\partial D \to \R^d$, which satisfies the condition
\begin{equation} \label{eq:gamma_inward}
    \gamma(y)^\top \nabla \rho_D(y) < 0, \qquad \forall\ y \in \partial D.
\end{equation}
This is an inward-pointing condition on $\gamma$, since $-\nabla \rho_D(y)/\|\nabla \rho_D(y)\|$ is the inward-pointing unit normal vector at $y \in \partial D$. Processes $(Y,\Phi)$ solve the RSDE \eqref{eq:RSDE} on $D$ with oblique reflection if the relationship \eqref{eq:RSDE} is satisfied with a reflecting term $\Phi$ which is of finite variation, carried on the set $\{t \geq 0: Y_t \in \partial D\}$, and inward-pointing in the direction specified by $\gamma$; that is, the reflection term is required to satisfy
\begin{equation} 
\label{eq:Phi_oblique} 
\Phi = \int_0^\cdot \gamma(Y_t)\,d\TV{\Phi}_t,
\end{equation}
where $\TV{\cdot}_t$ denotes the total variation up to time $t$, as in \eqref{eq:TV}.

\subsection{Brownian Motion with normal reflection on the unit ball} \label{sec:unit_ball}
To illustrate our approach we consider first the case of Reflected Brownian Motion on the unit ball $D = \{y \in \R^d: \|y\| \leq 1\}$ with normal reflection. This corresponds to $b \equiv 0$, $\sigma \equiv I_d$ (the $d \times d$ identity matrix),  $\rho_D(y) = \|y\|$, and $\gamma(y) = - y$. That is, $Y = W + \Phi$
for a standard $d$-dimensional Brownian Motion $W$ and reflecting process $\Phi$ satisfying
\begin{equation} \label{eq:Phi_radial}
\Phi = -\int_0^\cdot Y_t\,d\TV{\Phi}_t.
\end{equation} 

The curvature of the boundary makes it difficult to establish a folding representation where the diffusion $X$ only exhibits additive noise, so we will introduce a form of multiplicative noise.
In the present setting, the key idea for establishing the folding representation is to derive first an autonomous scalar RSDE for the norm of $Y$. 

Indeed, it is well-documented that the norm of a standard $d$-dimensional Brownian Motion is a Bessel process of order $d$ ($\mathrm{BES}(d)$). Similar calculations applied to the Reflected Brownian Motion $Y$ on the unit ball show that its radial part $R = \rho_D(Y) = \|Y\|$ satisfies
\[
    dR_t = \frac{d-1}{2R_t}dt + d\widetilde W_t + d\widetilde \Phi_t,
\]
with the scalar Brownian Motion $\widetilde W = \int_0^\cdot \frac{1}{R_t}Y_t^\top dW_t$ and with the reflection term 
\[\widetilde \Phi_t = \int_0^t\frac{1}{R_s}Y_s^\top d\Phi_s = - \int_0^t R_sd\TV{\Phi}_s = -  \int_0^t d\TV{\Phi}_s = - \TV{\Phi}_t, \qquad t \geq 0.\]
Here we used \eqref{eq:Phi_radial} in the penultimate equality, and the fact that $\Phi$ is supported on the set $\{t\geq 0: R_t = 1\}$ to obtain the final equality. It follows that $R$ is a BES($d$) process reflected at $r = 1$. The conclusions of Theorem~\ref{thm:main_bounded_interval} suggest that we have the representation 
\begin{equation} \label{eq:ball_norm_folding}
    R_t = |X_t^0-\Even(X_t^0)|, \qquad 0 \le t < \infty
\end{equation} for the radial part $R = \|Y\|$ of Brownian Motion in $\R^d$ reflected on the unit ball, where the scalar process $X^0$ satisfies
\begin{equation} \label{eq:signed_Bessel}
\begin{aligned} dX_t^0 & = \sign\big(X_t^0 - \Even(X_t^0)\big)\frac{d-1}{2|X^0_t-\Even(X^0_t)|}dt + dB^0_t \\
& = \bigg(\frac{d-1}{2X^0_t}1_{\{X_t^0 \in (0,1)\}} - \frac{d-1}{2(2-X^0_t)}1_{\{X^0_t \in [1,2)\}}\bigg)dt + dB^0_t
\end{aligned}
\end{equation}
on $[0,\xi)$, with $B^0$ standard Brownian Motion and $\xi = \inf\{t \geq 0: X_t^0 \in \{0,2\}\} = \inf\{t \geq 0: R_t = 0\}$. Here,
without loss of generality, we considered an initial condition $X_0^0 \in (0,2)$. Loosely speaking, the dynamics of $X^0$ coincide with those of a $\mathrm{BES}(d)$ process while $X^0$ is in the interval $(0,1]$, whereas the dynamics of $2-X^0$ coincide with those of a $\mathrm{BES}(d)$ process while $X^0$ is in the interval $[1,2)$. From standard properties of the Bessel process we can deduce that $\xi = \infty$ almost surely (since $d \ge 2$), so that the dynamics \eqref{eq:signed_Bessel} hold globally. As such, Theorem~\ref{thm:main_bounded_interval}, together with Remark~\ref{rem:bounded_drift}, guarantees the validity of the representation \eqref{eq:ball_norm_folding}.

With these preparations at hand, we are able to construct now a folding representation for Reflected Brownian Motion on the unit ball. To this end, note that the process $X = Y/R$ takes values in the unit sphere $S^{d-1} = \{z \in \R^d: \|z\| = 1\}$ and has dynamics
\begin{align*} dX_t & = \frac{1}{R_t}dY_t - \frac{Y_t}{R_t^2}dR_t - \frac{1}{R_t^2}d[Y,R]_t + \frac{Y_t}{R_t^3}d[R]_t  \\
& = -\frac{d-1}{2R_t^2}X_t dt + \frac{1}{R_t}(I_d - X_tX_t^\top )dW_t, 
\end{align*}
where the reflection terms cancel since $R=1$ on the support of $\Phi$.
In particular, no contribution from the reflection term $\Phi$ remains. In fact, $X$ is a time-changed spherical Brownian Motion. 

To remove the degeneracy of the dispersion term, which will be useful in the less explicit analysis of Subsection~\ref{sec:convex_oblique}, we project the dispersion coefficient and introduce a $(d-1)$-dimensional driving Brownian Motion $\overline W = (\overline W^1,\dots,\overline W^{d-1})$. Concretely, we take a projection matrix $P:S^{d-1} \to \R^{(d-1)\times d}$, which satisfies
\[P(x)P(x)^\top = I_{d-1}, \qquad (I_d - xx^\top)P(x)^\top P(x) = I_d-xx^\top, \qquad \forall x \in S^{d-1}.\]
For instance, one can take $P(x)$ to be the first $d-1$ rows of the matrix $I_d - 2ww^\top$ with $w = \frac{x-e_d}{\|x-e_d\|}$ (for less explicit dispersion matrices than $I_d - xx^\top$, $P(x)$ can be computed using, for instance, $QR$ decomposition). Then we have that $X$ satisfies
\[dX_t = -\frac{d-1}{2R_t^2}X_tdt + \frac{1}{R_t}(I_d - X_tX_t^\top)P(X_t)^\top d\overline W_t, \qquad \text{with  } \overline W_t = \int_0^\cdot P(X_t)dW_t\]
a $(d-1)$-dimensional standard Brownian Motion. These observations lead us to a folding representation, which is the content of the next proposition.
\begin{prop}
\label{prop:RBM_ball}
    Consider the $(d+1)$-dimensional SDE system for $(X^0,X) = (X^0,X^1,\dots, X^d)$:
    \begin{equation}
        \label{eq:SDE_ball}
    \begin{aligned}
        dX^0_t & = \bigg(\frac{d-1}{2X_t^0}1_{\{X_t^0 \in (0,1)\}} - \frac{d-1}{2(2-X_t^0)}1_{\{X_t^0 \in [1,2)\}}\bigg)dt + dB^0_t, & t \geq 0, \\
        dX_t & = - \frac{d-1}{2(X_t^0-\Even(X_t^0))^2}X_t\, dt + \frac{1}{|X_t^0-\Even(X_t^0)|}(I_d - X_tX_t^\top)P(X_t)^\top dB_t, & t \geq 0, 
    \end{aligned}
    \end{equation}
where $(B^0,B) = (B^0,B^1,\dots,B^{d-1})$ is a standard $d$-dimensional Brownian Motion and we initiate the process at any $X_0^0 \in (0,2)$ and $(X_0^1,\dots,X_0^d) \in S^{d-1}$. Then the SDE \eqref{eq:SDE_ball} has a weak solution, which is unique in law; the process $X$ takes values in the unit sphere $S^{d-1}$; whereas the process
\begin{equation} \label{eq:Y_oblique_ball}
Y = |X^0-\Even(X^0)|X = \big(X^01_{\{X^0 \in (0,1)\}} + (2-X^0)1_{\{X^0 \in [1,2)\}}\big)X 
\end{equation} is a Reflected Brownian Motion on the unit ball with normal reflection. \smallskip

In particular, the process $Y$ of \eqref{eq:Y_oblique_ball} has initial value $Y_0 =(X^0_01_{\{X^0_0 \in (0,1)\}} + (2-X^0_0)1_{\{X^0_0 \in [1,2)\}})X_0$ and
satisfies the RSDE \eqref{eq:RSDE} with $b \equiv  0$, $\sigma \equiv I_d$, driving
standard Brownian Motion 
\[ W  = \int_0^\cdot \bigg(X_t \sign\big(X_t^0-\Even(X_t^0)\big)dB^0_t + (I_d - X_tX_t^\top)P(X_t)^\top dB_t\bigg),
\]
and reflection term $
\Phi   = -\int_0^\cdot X_t dL^1_t(X^0) = - \int_0^\cdot Y_t dL^1_t(X^0)$.
\end{prop}
Proposition~\ref{prop:RBM_ball} will follow as a special case of Theorem~\ref{thm:convex} below, so we do not provide a separate proof for it.

\subsection{Oblique reflection on general convex domains} \label{sec:convex_oblique}
We return now to the setting of a general convex domain and invoke Assumption~\ref{ass:domain}.  
We start by stating our assumptions on the coefficients $(b,\sigma,\gamma)$, which specify the RSDE \eqref{eq:RSDE} and the direction of oblique reflection \eqref{eq:Phi_oblique}.

\begin{assum} \label{ass:coefficients_convex}
 We assume the following:
\begin{enumerate}[label = (\roman*), noitemsep]
    \item $b$ is Lipschitz continuous and $\sigma$ is of class $C^2$,
    \item  $a = \sigma \sigma^\top$ is uniformly elliptic, and 
    \item $\gamma$ is of class $C^2$, $\|\gamma(y)\| = 1$ for all $y \in \partial D$, and the inward-pointing condition \eqref{eq:gamma_inward} holds.
\end{enumerate}
\end{assum}
We start with a preparatory result, which will play an important role in the proof of Theorem~\ref{thm:convex} to come. The proof is postponed to Appendix~\ref{sec:polar_proof}.

\begin{lem} \label{lem:polar}
Let Assumptions~\ref{ass:domain} and \ref{ass:coefficients_convex} hold and suppose $Y$ is a solution to the RSDE \eqref{eq:RSDE} with oblique reflection \eqref{eq:Phi_oblique} on some stochastic time interval $[0,\xi)$. Then, as long as $Y_0 \ne 0$, we have
\[\P\big(Y_t = 0 \text{ for some } t \in [0,\xi)\big) = 0.\]
\end{lem}

In the analysis to come we will need to extend $\gamma$ to $\R^d \setminus \{0\}$. There are many ways this can be done, but to respect the geometry of the domain we choose the following extension
\begin{equation} \label{eq:gamma_extension}
\gamma(y) = \gamma\bigg(\frac{y}{\rho_D(y)}\bigg), \quad y \in \R^d \setminus \{0\}.
\end{equation}
Motivated by the case of Reflected Brownian Motion on the unit ball developed in Subsection~\ref{sec:unit_ball}, we look for a transformation of the reflected diffusion $Y$ satisfying \eqref{eq:RSDE} with oblique reflection \eqref{eq:Phi_oblique}, which consists of a scalar process $R$ with reflection, and of a multivariate process $X$ devoid of reflection terms.
To this end, we set 
\begin{equation} \label{eq:R,X}
R = \rho_D(Y), \qquad X = h(Y), 
\end{equation}
where $R$ will have a reflection term in its dynamics, while 
the map $h: D\setminus \{0\} \to \partial D$ will be chosen so that $X$ is an It\^o process with state space $\partial D$. To derive conditions on $h$ we compute for every $i=1,\dots,d$,
\begin{align}
    dX^i_t & = \nabla h^i(Y_t)^\top dY_t + \tfrac{1}{2}\mathrm{Tr}\big(\nabla^2 h^i(Y_t)d[Y]_t\big) \nonumber \\
    & = \Big(\nabla h^i(Y_t)^\top b(Y_t) + \tfrac{1}{2}\mathrm{Tr}\big(\nabla^2h^i(Y_t)a(Y_t)\big)\Big)dt + \nabla h^i(Y_t)^\top\sigma(Y_t)dW_t + \nabla h^i(Y_t)^\top d\Phi_t. \label{eq:X_convex_initial}
\end{align}
Using the oblique reflection condition \eqref{eq:Phi_oblique}, we see that the reflection term in \eqref{eq:X_convex_initial} vanishes if
\begin{equation} \label{eq:oblique_boundary} \nabla h^i(y)^\top \gamma(y) = 0 \qquad \text{holds for every } i=1,\dots,d \quad \text{and} \quad   y \in \partial D.
\end{equation}
Moreover, we require that $h(y) \in \partial D$ for every $y \in D\setminus \{0\}$. As such, to reconstruct $Y$ from $R$ and $X$ we will associate to any point $y \in \R^d \setminus \{0\}$ a curve $t\mapsto \eta_t(y)$ mapping $y$ to the origin. We write $T(y)$ for the hitting time of the origin when the curve $\eta$ is initiated at $y \in \R^d \setminus \{0\}$; that is, $\eta_{T(y)}(y) = 0$. Since, for any $y \in \partial D$, we will use $h$ to map the set 
\begin{equation} \label{eq:constancy_set}
\Gamma(y) := \big\{\eta_t(y): t \in \big[0,T(y)\big)\big\}
\end{equation} to $y$, this means that the gauge function $\rho_D$ of \eqref{eq:gauge} may be nonconstant on $\Gamma(y)$, but the functions $h^i$ for $i=1,\dots,d$ may not. This leads to the constancy condition 
\begin{equation} \label{eq:constancy_condition}
    0 = \frac{d}{dt} h^i\big(\eta_t(y)\big) = \nabla h^i\big(\eta_t(y)\big)^\top \dot \eta_t(y), \qquad y \in \R^d\setminus\{0\}, \quad t \in \big[0,T(y)\big). 
\end{equation}

The two conditions \eqref{eq:oblique_boundary} and \eqref{eq:constancy_condition} motivate the flow equation
\begin{equation} \label{eq:flow_oblique}
    \dot \eta_t(y) = \gamma\big(\eta_t(y)\big), \quad \eta_0(y) = y;  \qquad y \in \R^d\setminus\{0\}, \quad t \in \big[0,T(y)\big),
\end{equation}
which is posed on $\R^d \setminus \{0\}$ using the extension \eqref{eq:gamma_extension} for $\gamma$. We now establish the existence of a flow map $\eta$ satisfying \eqref{eq:flow_oblique} as well as some of its properties. The proof is lengthy, so we defer it to Appendix~\ref{sec:convex_proof}.
\begin{lem} \label{lem:flow_convex} Let Assumptions~\ref{ass:domain} and \ref{ass:coefficients_convex} hold. Then,
    \begin{enumerate}[label = (\roman*),noitemsep]
        \item \label{item:convex_flow} for every $y \in \R^d\setminus\{0\}$ the flow equation \eqref{eq:flow_oblique} has a unique solution $t \mapsto \eta_t(y)$ on a maximal time interval $[0,T(y))$. We have $T(y) < \infty$ and  $\eta_t(y) \to 0$ as $t \uparrow T(y)$ for all $y \in \R^d\setminus\{0\}$; 
        \item \label{item:bijectivity} every $y \in D^* := D\setminus \{0\}$ can be uniquely written as $y = \eta_t(z)$ for some $z \in \partial D$ and $t \in [0,T(z))$. That is, the map
        \[\eta:\mathcal{U} \to D^* \quad \text{given by} \quad (z,t) \mapsto \eta_t(z), \quad \text{where} \quad \mathcal{U} = \{(z,t): z \in \partial D,\,  0 \le t < T(z)\}\]
        is a bijection. As such, we can write $\eta^{-1}(y) = (h(y),\theta(y))$ for maps $h:D^* \to \partial D$ and $\theta:D^* \to [0,\infty)$;
        \item \label{item:flow_regularity} the flow map $(y,t) \mapsto \eta_t(y)$ and the inverse maps $h$ and $\theta$ are all of class $C^2$. 
    \end{enumerate}
\end{lem}

With $h$ as in Lemma~\ref{lem:flow_convex}\ref{item:bijectivity}, it is easy to see that $h(y) = z$ for any $y \in \Gamma(z)$, where $\Gamma(\cdot)$ is the set defined in \eqref{eq:constancy_set}, so that \eqref{eq:oblique_boundary} holds courtesy of \eqref{eq:constancy_condition} and \eqref{eq:flow_oblique}. \smallskip

We now turn our attention to reconstructing the reflected process $Y$ from the boundary process $X$ and the radial process $R$. From observing $X_t$, it is clear that $Y_t \in \Gamma(X_t)$, so it just remains to determine the precise point of this set at which $Y_t$ lies; this is  equivalent to determining a time $\tau$ so that $Y_t = \eta_\tau(X_t)$, where $\eta$ is the solution to \eqref{eq:flow_oblique} obtained in Lemma~\ref{lem:flow_convex}. As the gauge function measures how far a point lies between the boundary and the  origin, the value $R_t$ precisely pins down how far along the curve $s \mapsto \eta_s(X_t)$ the value $Y_t$ is. This observation yields the reconstruction formula
    \begin{equation} \label{eq:Y_reconstruct}
        Y_t = \eta_{\tau(R_t,X_t)}(X_t), \qquad \text{where } \tau(r,x)  \text{ is the unique } s \geq 0 \text{ such that } \rho_D\big(\eta_s(x)\big) = r.
    \end{equation}
The map $t \mapsto \rho_D(\eta_t(x))$ is strictly decreasing with rate uniformly bounded away from zero (see equation \eqref{eq:inward_pointing_estimate} obtained during the proof of Lemma~\ref{lem:flow_convex}). Additionally, since $\eta_t(x) \to 0$ as $t \uparrow T(x)$, we have that $\rho_D(\eta_t(x)) \downarrow 0$ as $t \uparrow T(x)$. Since $\rho_D(\eta_0(x)) = \rho_D(x) = 1$ for all $x \in \partial D$, we deduce from these observations that $\tau(r,x)$ is well-defined for all $(r,x) \in (0,1] \times \partial D$. 

We now seek to derive autonomous dynamics for $(X,R)$ and establish that $Y$, which satisfies \eqref{eq:RSDE} on $D$ with $\Phi$ given by \eqref{eq:Phi_oblique}, can  be reconstructed. We introduce the map 
\begin{equation} \label{eq:H} 
H(r,x) = \eta_{\tau(r,x)}(x), \qquad (r,x) \in (0,1] \times \partial D,
\end{equation}
in terms of which the reconstruction formula \eqref{eq:Y_reconstruct} becomes $Y_t = H(R_t,X_t)$. Whereas, 
using the function $h$ obtained in Lemma~\ref{lem:flow_convex}\ref{item:bijectivity}, we see that the dynamics for $X$ in \eqref{eq:X_convex_initial} become 
\[     dX_t = \alpha_X(R_t,X_t)dt + \sigma_X(R_t,X_t)dW_t,
\]
where $\alpha_X: (0,1] \times \partial D \to \R^d$ and $\sigma_X: (0,1] \times \partial D \to \R^{d \times d}$ are given by
\begin{align} \label{eq:alphaX}
    \alpha_X^i(r,x) & = \nabla h^i\big(H(r,x)\big)^\top b\big(H(r,x)\big) + \tfrac{1}{2}\mathrm{Tr}\Big(\nabla^2 h^i\big(H(r,x)\big)a\big(H(r,x)\big)\Big), & i=1,\dots,d, \\
   \sigma_X(r,x)  & = J_h\big(H(r,x)\big) \sigma\big(H(r,x)\big), \nonumber  
\end{align}
with $J_h(\cdot)$ the Jacobian of $h$. From \eqref{eq:constancy_condition}, we see that $J_h$ has rank $d-1$ since it degenerates along the direction of the flow $\eta$. As such, $\sigma_X$ also has rank $d-1$, and we can find a projection matrix $P:D^* \to \R^{(d-1)\times d}$ satisfying 
\begin{equation}
    \label{eq:projection}
    P(y) P(y)^\top  = I_{d-1} \quad \text{and} \quad\sigma_X(r,x)P(y)^\top P(y) = \sigma_X(r,x) 
\end{equation}
for all $(r,x) \in (0,1] \times \partial D$, and $y = H(r,x)$. Then the matrix
\begin{equation}
    \label{eq:sigmaX_bar} \sigma_X^P(r,x) = \sigma_X(r,x)P\big(H(r,x)\big)^\top
\end{equation} has rank $d-1$ for every $(r,x) \in (0,1] \times \partial D$. With this dimension reduction we can write
\begin{equation} 
dX_t = \alpha_X(R_t,X_t)dt +  \sigma_X^P(R_t,X_t)d\overline W_t,
\label{eq:X_convex_dynamics}
\end{equation}
where $\overline W = \int_0^\cdot P(Y_t)dW_t$ is a standard $(d-1)$-dimensional Brownian Motion. \smallskip

Next, we compute the dynamics of $R = \rho_D(Y)$ as in \eqref{eq:R,X}, using It\^o's formula:    \begin{align} dR_t & = \Big(\nabla \rho_D(Y_t)^\top b(Y_t) + \tfrac{1}{2}\mathrm{Tr}\big(\nabla^2 \rho_D(Y_t)a(Y_t)\big)\Big)dt + \nabla \rho_D(Y_t)^\top \sigma(Y_t)dW_t + \nabla \rho_D(Y_t)^\top d\Phi_t \nonumber  \\
& = \widetilde b(R_t,X_t)dt + \widetilde \sigma(R_t,X_t)d\widetilde W_t +d\widetilde \Phi_t. \label{eq:Rt_dynamics}
\end{align}
In the final equality, we recalled $Y_t = H(R_t,X_t)$ and introduced the drift and dispersion coefficients
\begin{align} \label{eq:widetilde_b}
\widetilde b(r,x) & = \nabla \rho_D\big(H(r,x)\big)^\top b\big(H(r,x)\big) + \tfrac{1}{2}\mathrm{Tr}\Big(\nabla^2 \rho_D\big(H(r,x)\big)a\big(H(r,x)\big)\Big),  \\ \label{eq:widetilde_sigma}
\widetilde \sigma(r,x) & = \sqrt{\nabla \rho_D\big(H(r,x)\big)^\top a\big(H(r,x)\big)\nabla \rho_D\big(H(r,x)\big)}
\end{align}
 for $(r,x) \in (0,1]  \times \partial D$, the
scalar Brownian Motion \begin{align*}
\widetilde W & = \int_0^\cdot \frac{\nabla \rho_D(Y_t)^\top \sigma(Y_t)}{\sqrt{\nabla \rho_D(Y_t)^\top a(Y_t)\nabla \rho_D(Y_t)}}dW_t ,
\intertext{and the scalar finite variation term} 
    \widetilde \Phi & = \int_0^\cdot \nabla \rho_D(Y_t)^\top \gamma(Y_t)d\TV{\Phi}_t.
\end{align*} 
Note that $\widetilde \Phi$ is supported on the set $\{t \geq 0: Y_t \in \partial D\} = \{t \geq 0: R_t = 1\}$ and inward-pointing, since $\nabla \rho_D(y)^\top \gamma(y) < 0$ for every $y \in \partial D$ by \eqref{eq:gamma_inward}. \smallskip

From \eqref{eq:Rt_dynamics}, we see that, loosely speaking, the process $R$ is a one-dimensional reflected diffusion on $(0,1]$ coupled with the unreflected diffusion $X$ of \eqref{eq:X_convex_dynamics}. As such, from Theorem~\ref{thm:interval_Z} with $(R,X)$ in place of $(Y,Z)$, we expect to be able to write $R = f(|X^0-\Even(X^0)|,X)$ using a function $f$ of the form \eqref{eq:nu_Z} for an appropriate function $\widehat \sigma$. This leads us to the main result of this section, which establishes a general folding representation for diffusions obliquely reflected on convex domains. The proof of this result is contained in Appendix~\ref{sec:main_proof}.
\begin{theorem} \label{thm:convex}
    Let Assumptions~\ref{ass:domain} and \ref{ass:coefficients_convex} hold. 
    \begin{enumerate}[label = (\roman*),noitemsep]
        \item \label{item:folding_regularity} The function $H:(0,1] \times \partial D \to D^*$ defined in \eqref{eq:H} with $\eta$ as in \eqref{eq:flow_oblique}, and the function $\tau:(0,1] \times\partial D \to [0,\infty)$ defined in \eqref{eq:Y_reconstruct},  are of class $C^2$. Moreover, the relationships
        \begin{equation}
    \label{eq:inverse_identities}
        \rho_D\big(H(r,x)\big) = r, \qquad h\big(H(r,x)\big) = x  
        \end{equation}
        hold for $(r,x) \in (0,1] \times \partial D$, where $h$ is given in Lemma~\ref{lem:flow_convex}\ref{item:bijectivity}; that is,  $H$ is invertible with inverse $H^{-1} = (\rho_D,h)$;
        \item \label{item:SDE_system} For $r \in (0,1]$,  $x^0 \in (0,2)$ and $x \in \partial D$ set 
        \begin{equation} \label{eq:SDE_system_coefficients}
        \begin{aligned}
        \rho(r,x) & = \frac{P(H(r,x))\sigma^\top(H(r,x))\nabla
        \rho_D(H(r,x))}{\sqrt{\nabla
        \rho_D(H(r,x))^\top a(H(r,x))\nabla
        \rho_D(H(r,x))}}, \qquad \widehat \sigma(r,x) = \widetilde \sigma(r,x)\sqrt{1-\|\rho(r,x)\|^2}, \\
        \nu(x) & = \bigg(\int_0^1 \frac{du}{\widehat \sigma(u,x)}\bigg)^{-1}, \qquad 
        f(r,x)  = \kappa^{-1}\bigg(\frac{r}{\nu(x)},x\bigg), \quad \text{where} \quad  \kappa(\xi,x) = \int_0^\xi \frac{du}{\widehat \sigma(u,x)}, \\
           \overline x^0 & = x^01_{\{x^0 \in (0,1)\}}  - (2-x^0)1_{\{x^0 \in [1,2)\}},\\
            \overline \alpha(x^0,x) & = \alpha_X\big(f(|\overline x^0|,x),x\big), \quad 
            \overline \sigma(x^0,x) = \sigma_X^P \big(f(|\overline x^0|,x),x\big), \quad \overline a(x^0,x) = \overline \sigma(x^0,x)\overline \sigma(x^0,x)^\top, \\
               \overline \sigma^0(x^0,x) & = \sign(\overline x^0)\bigg(\frac{\rho(f(|\overline x^0|,x),x)}{\sqrt{1-\|\rho(f(|\overline x^0|,x),x)\|^2}} -\frac{ \overline \sigma(x^0,x)^\top \nabla_xf(|\overline x^0|,x)}{\partial_r f(|\overline x^0|,x)}\bigg), \\
            \zeta(x^0,x) & = \frac{1}{\partial_r f(|\overline x^0|,x)}\Big(\widetilde b\big(f(|\overline x^0|,x),x\big) - \tfrac{1}{2}\nu(x)^2\partial_{rr} f(|\overline x^0|,x)\big(1+\|\overline \sigma^0(x^0,x)\|^2\big) \\ 
& \qquad \qquad \qquad \qquad - \nabla_x f(|\overline x^0|,x)^\top \overline \alpha( x^0,x)  - \tfrac{1}{2}\mathrm{Tr}\big(\nabla^2_x f(|\overline x^0|,x)\overline a(x^0,x)\big) \\
& \qquad \qquad \qquad \qquad- \partial_r \nabla_x f(|\overline x^0|,x)^\top \overline  \sigma(x^0,x)\overline \sigma^0(x^0,x)\Big),\\
            \overline \alpha^0(x^0,x) & = \sign(\overline x^0)\zeta(x^0,x),
        \end{aligned} 
        \end{equation}
        where $(\alpha_X,P, \sigma_X^P)$ and $(\widetilde b, \widetilde \sigma)$ are as in  \eqref{eq:alphaX}--\eqref{eq:sigmaX_bar} and \eqref{eq:widetilde_b}--\eqref{eq:widetilde_sigma} respectively and, as in Theorem~\ref{thm:interval_Z}, $\kappa^{-1}$ refers to the inverse of $\kappa$ in its first argument.
        
Consider the SDE system for $(X^0,X) = (X^0,X^1,\dots,X^d)$ given by
        \begin{equation}
        \label{eq:SDE_sytem_convex}
        \begin{aligned}
        dX^0_t & = \overline \alpha^0(X^0_t,X_t)dt + \nu(X_t)\big(dB^0_t + \overline \sigma^0(X^0_t,X_t)^\top dB_t\big),\\
        dX_t  & =   \overline \alpha(X^0_t,X_t)dt + \overline \sigma(X^0_t,X_t)dB_t,
        \end{aligned}
        \end{equation}
        where $(B^0,B) = (B^0,B^1,\dots,B^{d-1})$ is a standard $d$-dimensional Brownian Motion. This system has a weak solution, unique in law, for any initial condition $(X^0_0,X_0) \in (0,2) \times \partial D$; 
        \item \label{item:convex_folding} The process 
        \[Y = H(R,X) \qquad \text{with} \qquad  R = f(|\overline X^0|,X)\]
        satisfies the RSDE \eqref{eq:RSDE} with oblique reflection on $D$ in the manner of \eqref{eq:Phi_oblique}, initial condition $Y_0 = H(R_0,X_0)$, driving Brownian Motion
        \begin{equation} \label{eq:W_convex}
            \begin{aligned}
                 W & = \int_0^\cdot \sign(\overline X^0_t)\widehat\sigma(R_t,X_t)\frac{\sigma^{-1}(H(R_t,X_t))\gamma(H(R_t,X_t))}{\nabla \rho_D(H(R_t,X_t))^\top \gamma(H(R_t,X_t))}dB^0_t +\int_0^\cdot P\big(H(R_t,X_t)\big)^\top dB_t,
            \end{aligned}
        \end{equation} and reflection term 
        \begin{equation} \label{eq:Phi_convex} 
        \Phi = -\int_0^\cdot \frac{\widehat \sigma(1,X_t)\gamma(H(1,X_t))}{\nu(X_t)\nabla \rho_D(H(1,X_t))^\top\gamma(H(1,X_t))} dL_t^1(X^0).
    \end{equation}    \end{enumerate}
\end{theorem}

 In general, the flow equation \eqref{eq:flow_oblique} does not admit an explicit solution and, consequently, the functions $H$, $h$ and $\tau$ do not either. The exception, which works on a generic convex domain $D$ satisfying Assumption~\ref{ass:domain}, is the case of inward radial reflection. Concretely, suppose that $\gamma(y) = -y/\|y\|$ for all $y \in \partial D$ and note that this choice
        satisfies \eqref{eq:gamma_inward} by the Euler identity $y^\top \nabla \rho_D(y) = \rho_D(y)$. Then the flow equation \eqref{eq:flow_oblique} has the explicit solution
        \[\eta_t(y) = \frac{\|y\|-t}{\|y\|}y, \qquad 0 \leq t < \|y\| = T(y).\]
        The homogeneity of $\rho_D$ shows that for any $(r,x) \in (0,1] \times \partial D$ we have $\rho_D(\eta_t(x)) = r \iff t = \|x\|(1- r)$. As such, it follows that $\tau(r,x) = \|x\|(1-r)$ and, consequently, that $H(r,x) = rx$. In the case of the unit ball, this  coincides precisely with the discussion in Subsection~\ref{sec:unit_ball}; in particular, Theorem~\ref{thm:convex} reduces then to Proposition~\ref{prop:RBM_ball}.

    \subsubsection*{Funding}

    Ioannis Karatzas gratefully acknowledges support from the National Science Foundation under Grant DMS-25-06199, and from a Lenfest Award at Columbia University.
\appendix 
\section{Proofs of Lemmas~\ref{lem:ODE_quadrant}, \ref{lem:polar} and \ref{lem:flow_convex}}

\subsection{Proof of Lemma~\ref{lem:ODE_quadrant}} \label{sec:orthant_proof}
    With the exception of $f$ being a member of $W^{2,\infty}_{\mathrm{loc}}(D)$ and having range $D$, which leads to the necessity of \eqref{eq:sigma_offdiagonal_boundary_condition}, the remaining claims have already been established as part of the discussion in Subsection~\ref{sec:orthant}. Note that sufficiency of \eqref{eq:sigma_offdiagonal_boundary_condition} is immediate, since on $\{y^i = 0\}$ it gives $\partial_j f^i = \sigma^{ij}(f) = 0$ for $j \neq i$; this way, $f^i$ stays at its initial value $f^i(0) = 0$, establishing \eqref{eq:f_boundary_quadrant}. 
    
    The claim $f \in W^{2,\infty}_{\mathrm{loc}}(D)$ follows directly from \eqref{eq:nonlinear_PDE_quadrant} and the global Lipschitz continuity of $\sigma$. To establish surjectivity of $f$, we define for any $y \in D$ the line-segment $L_y = \{sy: 0 \leq s \leq 1\}$ connecting $y$ to the origin.
    Next, we introduce the set
    \[E := \{y \in D: f \text{ has a $C^1$ inverse on a neighborhood $U_y$ of $L_y$ in $D$}\}. \]
    Since $J_f(y) = \sigma(f(y))$ is invertible by our assumption on $\sigma$, we have from the Inverse Function Theorem that $f$ is a local $C^1$ diffeomorphism. From here, it follows that the set $E$ is open in $D$. Indeed, the line segments $L_{y'}$ vary continuously with their endpoint, so for $y'$ in a small neighborhood of $y$ we have $L_{y'} \subset U_y$, and the $C^1$ inverse on $U_y$ shows that $y' \in E$. Additionally, since $f(0) = 0$ and $f$ is a local diffeomorphism at $0$, we  have that $0 \in E$, so that $E$ is a nonempty set. Hence, if we can show that $E$ is also closed in $D$, we will be able to conclude that $E=D$, which will establish the claim.

    To this end, let $\{y_n\}_{n \in \mathbb{N}} \subset E$ be a sequence that converges to some $y_\infty \in D$. Since each $y_n$ is in $E$, the function $h_n(s) = f^{-1}(sy_n)$ is well-defined, satisfies $h_n(0) = 0$, and is of class $C^1([0,1])$. Differentiating $sy_n = f(h_n(s))$ yields by the chain rule and \eqref{eq:nonlinear_PDE_quadrant}, 
    \[y_n = \frac{d}{ds} f\big(h_n(s)\big) =  J_f\big(h_n(s)\big)\frac{d}{ds}h_n(s) = \sigma(sy_n)\frac{d}{ds}h_n(s) \qquad \implies \frac{d}{ds}h_n(s) = \sigma^{-1}(sy_n)y_n.\] 
    Since $\sigma^{-1}$ is continuous on the compact set  $\overline L = \overline{(\cup_n L_{y_n}) \cup L_{y_\infty}}$, the operator norm of $\sigma^{-1}$ is bounded by some constant $M > 0$ on $\overline L$. By the Dominated Convergence Theorem we conclude that 
    \[\lim_{n \to \infty} h_n(s) = \lim_{n \to \infty} \int_0^s \sigma^{-1}(uy_n)y_ndu = \int_0^s \sigma^{-1}(uy_\infty)y_\infty du =: h_\infty(s)\] 
    for every $s \in [0,1]$. By continuity of $f$, we see that $f(h_\infty(s)) = \lim_{n \to \infty} f(h_n(s)) = sy_\infty$.
    
    We are now ready to show that $y_\infty \in E$. For each $s \in [0,1]$, from the above analysis we have that $h_\infty(s)$ is a preimage of $sy_\infty$ and that $f$ is a local diffeomorphism at $h_\infty(s)$. As such, there exist open neighborhoods $U_s, V_s \subset D$ with $s y_\infty \in U_s$, $h_\infty(s) \in V_s$, which we may take to be connected, such that $f|_{V_s}: V_s \to U_s$ is a $C^1$ diffeomorphism. That is, there exists a local inverse $\phi_s: = (f|_{V_s})^{-1}$, which satisfies $\phi_s(sy_\infty) = h_\infty(s)$. The collection of sets $\{U_s\}_{s \in [0,1]}$ is an open cover of the compact set $L_{y_\infty}$, so we can extract a finite subcover $U_{y_\infty} := \cup_{n=1}^N U_{s_n}$. Whenever $s y_\infty \in U_{s_n} \cap U_{s_m}$, both $\phi_{s_n}(s y_\infty)$ and $\phi_{s_m}(s y_\infty)$ are preimages of $s y_\infty$ under $f$ lying on $L_{y_\infty}$, and each equals $h_\infty(s)$; thus $\phi_{s_n}$ and $\phi_{s_m}$ agree along $L_{y_\infty} \cap U_{s_n} \cap U_{s_m}$. Since $f$ is a local diffeomorphism, its local inverse is unique near any point of this set, so $\phi_{s_n}$ and $\phi_{s_m}$ agree on a neighborhood of $L_{y_\infty} \cap U_{s_n} \cap U_{s_m}$. Shrinking the $U_{s_n}$ to a sufficiently small neighborhood of $L_{y_\infty}$ if necessary, we obtain that the $\phi_{s_n}$ agree on the overlaps $U_{s_n} \cap U_{s_m}$. It follows that $f^{-1}:U_{y_\infty} \to D$ given by $f^{-1}(y) = \phi_{s_n}(y)$ if $y \in U_{s_n}$ is a well-defined $C^1$ inverse of $f$. This establishes that $y_\infty \in E$. Hence $E$ is closed, so $E = D$, and the local inverses patch to a global $C^1$ inverse of $f$ on $D$; in particular, $f$ is a bijection onto $D$. This completes the proof.

\subsection{Proof of Lemma~\ref{lem:polar}} \label{sec:polar_proof}
    Let $r> 0$ be small enough so that the ball centered at the origin with radius $2r$ is compactly contained in the interior of $D$. We inductively define a sequence of stopping times by setting $\theta_0 = \inf\{t \ge 0: \|Y_t\| \leq r\}$ and 
    \[\tau_k = \inf\{t \ge \theta_k: \|Y_t\| \ge 2r\}, \quad \theta_{k+1} = \inf\{t \geq \tau_k: \|Y_t\| \le r\} \qquad \text{for } k=0,1,\dots\]
Clearly, we have that
\begin{align*}
    \P\big(Y_t = 0 \text{ for some } t \in [0,\xi)\big) & = \P\big({\displaystyle \cup_k} 
    \{Y_t = 0 \text{ for some } t \in [\theta_k,\tau_k)\}\big) \\
    & \le \sum_{k=0}^\infty \P\big(Y_t =0 \text{ for some } t \in [\theta_k,\tau_k)\big).
\end{align*}
Hence, it suffices to show that $\P(Y_t = 0 \text{ for some } t \in [\theta_k,\tau_k)) = 0$ holds for arbitrary $k$. Since $\Phi$ only accumulates on the set $\{t \ge 0: Y_t \in \partial D\}$, we see that $\Phi$ is constant on $[\theta_k,\tau_k)$ for any $k$. As such, on this time interval, $Y$ satisfies the standard SDE
\[dY_t = b(Y_t)dt + \sigma(Y_t)dW_t, \qquad t \in [\theta_k,\tau_k).\]
We now let $X$ be the solution of the SDE $dX_t = b(X_t)dt + \sigma(X_t)dW_t$ on the same probability space, started at $X_{\theta_k} = Y_{\theta_k}$ and driven by the same Brownian Motion $W$. Since the coefficients $b$ and $\sigma$ are Lipschitz continuous and bounded we have a pathwise unique, strong solution to this SDE. As such, we deduce that $X_t = Y_t$ for all $t \in [\theta_k,\tau_k)$. However, since $d \ge 2$, $b$ and $\sigma$ are Lipschitz, and $a = \sigma \sigma^\top$ is uniformly elliptic it is known that $X$ does not hit points; in particular, it does not hit the origin (see \cite[Chapter~11, Theorem~4.1]{friedman1975stochastic}). As such, we have that 
\[\P\big(Y_t = 0 \text{ for some } t \in [\theta_k,\tau_k)\big) = \P\big(X_t = 0 \text{ for some } t \in [\theta_k,\tau_k)\big) = 0,\]
which completes the proof.

\subsection{Proof of Lemma~\ref{lem:flow_convex}} \label{sec:convex_proof}

    We start by noting that the derivative of the gauge function $\nabla \rho_D$ is homogeneous of order zero and the same is true of $\gamma$ due to the chosen extension \eqref{eq:gamma_extension}. Hence, the map $y \mapsto \gamma(y)^\top \nabla \rho_D(y)$ on $\R^d \setminus \{0\}$ can be viewed as a map on $\partial D$ in terms of achievable values. Since $\partial D$ is compact this continuous map has a minimal value, and from \eqref{eq:gamma_inward} it follows that there exists a constant $c_0 >0$ such that 
    \begin{equation} \label{eq:gamma_uniform_estimate}
        \gamma(y)^\top \nabla \rho_D(y) \le -c_0 < 0, \qquad \forall y \in \R^d \setminus \{0\}.
    \end{equation}
    Now, because $\gamma$ is $C^2$ on $\R^d \setminus \{0\}$, it is locally Lipschitz continuous, which guarantees that \eqref{eq:flow_oblique} has a unique solution $\eta_t(y)$ on a maximal interval $[0,T(y))$. Since $\gamma$ is bounded, finite-time blowup is not possible, so the solution may exit the domain only if $\eta_t(y) \to 0$ as $t \uparrow T(y)$. 
    
    Fix $y \in \R^d \setminus \{0\}$ and define $r(t) = \rho_D(\eta_t(y))$ for $t \in [0,T(y))$. Note that $r(0) = \rho_D(y)$ and by the chain rule, \eqref{eq:flow_oblique}, and the estimate \eqref{eq:gamma_uniform_estimate} we have that 
    \begin{equation} 
\label{eq:inward_pointing_estimate}
    \dot r(t) = \nabla \rho_D\big(\eta_t(y)\big)^\top \dot \eta_t(y) = \nabla\rho_D\big(\eta_t(y)\big)^\top \gamma\big(\eta_t(y)\big) \le -c_0 < 0.
    \end{equation}
    As such, $r$ is strictly decreasing, which tells us that $y \in D^* \implies \eta_t(y) \in D^*$ for all $t \in [0,T(y))$.
    Moreover, the estimate \eqref{eq:inward_pointing_estimate} implies that $r(t) \le \rho_D(y) - c_0 t$, from which we deduce that $T(y) \le \rho_D(y)/c_0 < \infty$.  Since $\rho_D(y) = 0 \iff y = 0$,  we see that $\lim_{t \uparrow T(y)} r(t) = 0$ and, as such, $\lim_{t \uparrow T(y)} \eta_t(y) = 0$. This proves item \ref{item:convex_flow}.

    To establish \ref{item:bijectivity}, first consider the related flow equation  
    \[\dot \psi_t(y) = - \gamma\big(\psi_t(y)\big), \quad \psi_0(y) = y; \qquad y \in \R^d\setminus\{0\}, \quad t \in [0,\infty).\] For the same reasons as for \eqref{eq:flow_oblique}, this equation has a unique local flow map $\psi$. Defining $R(t) = \rho_D(\psi_t(y))$ for fixed $y \in \R^d \setminus \{0\}$, we readily obtain that $\dot R(t) \geq c_0 > 0$, so $R$ is strictly increasing and that $R(t) \ge \rho_D(y) + c_0t$. Since $\gamma$ is bounded, finite-time blowup is not possible, establishing that the map $t \mapsto \psi_t(y)$ is well-defined on $[0,\infty)$ for all $y \in \R^d \setminus \{0\}$. From the strictly increasing property of $R$ and the fact that $R(t) \to \infty$ as $t \to \infty$, we see that for every $y \in D^*$ there exists a unique $\theta(y) \in [0,\infty)$ such that $h(y) := \psi_{\theta(y)}(y) \in \partial D$. 
    
    We now establish that \begin{equation} \label{eq:flow_relationship} \eta_t\big(\psi_s(y)\big)  = \psi_{s-t}(y), \qquad \text{for all } 0 \leq t \leq s. 
    \end{equation} To see this, fix $s > 0$ and set $u_t(y) = \psi_{s-t}(y)$ for $t \in [0,s]$. Then we have that $u_0(y) = \psi_s(y)$ and 
    \[\dot u_t(y) = -\dot \psi_{s-t}(y) = \gamma\big(\psi_{s-t}(y)\big) = \gamma\big(u_t(y)\big).\]
    As such, $u$ solves the flow equation \eqref{eq:flow_oblique} initiated at $\psi_s(y)$ and, by uniqueness, must be equal to $\eta_t(\psi_s(y))$. In particular, given $y \in D^*$ we have that $(\psi_{\theta(y)}(y),\theta(y)) \in \mathcal{U}$. Hence, by taking $s = t = \theta(y)$ in \eqref{eq:flow_relationship} we have that $\eta_{\theta(y)}(\psi_{\theta(y)}(y)) = y$, which establishes surjectivity of $\eta$. 
    
    To obtain injectivity, assume that $\eta_{t_1}(z_1) = y = \eta_{t_2}(z_2)$ for some $y \in D^*$, $z_1,z_2\in \partial D$ and $t_i\in [0,T(z_i))$ for $i=1,2$. Set $v_t(y) = \eta_{t_1-t}(z_1)$ for $t \in [0,t_1]$. It is easy to see that $v_0(y) = y$ and
    $\dot v_t(y) = -\gamma(v_t(y))$,
    so that $v_t(y) = \psi_t(y)$ by uniqueness of solutions to the flow equation. However, we also clearly have $v_{t_1}(y) = z_1 \in \partial D$. From the strict increase of the function $R(t) = \rho_D(\psi_t(y))$ we know that, for any $y \in D^*$, $\psi_\cdot(y)$ hits $\partial D$ at the unique time $\theta(y)$. As such, it follows that $t_1 = \theta(y)$ and $z_1 = v_{t_1}(y) =  h(y)$. Arguing identically with $\eta_{t_2}(z_2)$ establishes that $t_2 = \theta(y)$ and $z_2 = h(y)$, whence $t_1 = t_2$ and $z_1 = z_2$. This completes the proof of item \ref{item:bijectivity}.

    Finally, to establish \ref{item:flow_regularity}, first note that the flow maps $(t,y)\mapsto \eta_t(y)$ and $(t,y) \mapsto \psi_t(y)$ are of class $C^2$ because $\gamma$ is of class $C^2$ (see, e.g., \cite[Theorem~5.4.1]{hartman1982ordinary}). To obtain smoothness of $\theta$, define the $C^2$ function $F(s,y) = \rho_D(\psi_s(y)) -1$ for $y \in D^*$ and $s \in [0,\infty)$. Note that for any $y \in D^*$,  $F(\theta(y),y) = \rho_D(h(y)) -1 = 0$ and that $\partial_s F(s,y) = \dot R(s) \geq c_0 > 0$. As such, the Implicit Function Theorem establishes that $\theta(y)$ is $C^2$. The function $h(y) = \psi_{\theta(y)}(y)$ is then also $C^2$ as a composition of $C^2$ maps. This completes the proof. 

\section{Proof of Theorem~\ref{thm:convex}} \label{sec:main_proof}

To prove the first item, we start by defining the function $G: \mathcal{U} \to (0,\infty)$,  via $G(x,t) = \rho_D(\eta_t(x))$, where  $\mathcal{U}$ is defined in Lemma~\ref{lem:flow_convex}\ref{item:bijectivity}. Since $\rho_D$ and $\eta$ are of class $C^2$, we see that $G$ is as well. Additionally, by the definition of $\tau$ we have that $G(x,\tau(r,x)) - r = 0 $ for all $(r,x) \in (0,1] \times \partial D$. The same computation as in \eqref{eq:inward_pointing_estimate} shows that $\partial_tG(x,t) \leq - c_0 < 0$ so by the Implicit Function Theorem it follows that $\tau$ is $C^2$. It is then clear that $H$ is also of class $C^2$ as a composition of the $C^2$ maps $\eta$ and $\tau$.
To obtain the identities in \eqref{eq:inverse_identities} we simply note that $H(r,x) = \eta_t(z)$ with $(z,t) = (x,\tau(r,x)) \in \mathcal{U}$ and that $\rho_D(H(r,x)) = r$ by definition of $\tau$. As such, the bijectivity of the flow map $\eta$ guaranteed by Lemma~\ref{lem:flow_convex}\ref{item:bijectivity} establishes \eqref{eq:inverse_identities} and completes the proof of \ref{item:folding_regularity}.

    We now turn our attention to the SDE system \ref{item:SDE_system}. Assumption~\ref{ass:coefficients_convex} on the coefficients ensures that the drift and dispersion coefficients $(\overline \alpha^0, \overline \sigma^0, \overline \alpha, \overline \sigma)$ are bounded on   $[\varepsilon,2-\varepsilon] \times \partial D$ for every $\varepsilon > 0$. However, the coefficients may blow up as $r = f(|\overline x^0|,x) \downarrow 0$ (see the dynamics \eqref{eq:SDE_ball} for the case of the unit ball). Additionally, as in the proof of Theorem~\ref{thm:interval_Z} the diffusion matrix for the system $(X^0,X)$ is given in block form as
\[A(x^0,x) = \begin{bmatrix}
    \nu(x)^2(1+\|\overline \sigma^0(x^0,x)\|^2) & \nu(x)\overline \sigma^0(x^0,x)^\top \overline \sigma(x^0,x)^\top  \\
    \nu(x)\overline \sigma(x^0,x) \overline \sigma^0(x^0,x) & \overline a(x^0,x)
\end{bmatrix}, \qquad (x^0,x) \in (0,2) \times \partial D.
\]
This diffusion coefficient is uniformly elliptic, but the block off-diagonal entries have the co-dimension one discontinuity set $\{x^0 = 1\}$ due to the appearance of $\sign(\overline x^0)$ in the definition of $\overline \sigma^0$ in \eqref{eq:SDE_system_coefficients}. As such, applying \cite[Remark~3.4]{krylov2004on} locally on $[\varepsilon,2-\varepsilon] \times \partial D$ and then sending $\varepsilon \to 0$ we obtain a unique weak solution to \eqref{eq:SDE_sytem_convex} on the interval $[0,\xi)$, where 
\[\xi = \inf\big\{t \ge 0: X^0_t \in \{0,2\}\big\} = \inf\{t \ge 0: R_t =0\}\]
is the explosion time and $R_t = f(|\overline X_t^0|,X_t)$. Next, we work towards establishing the folding representation of item \ref{item:convex_folding} and postpone establishing that $\xi = \infty$, $\P$-a.s.\ to the end of this proof.

To establish \ref{item:convex_folding}, we show first that $R$ satisfies \eqref{eq:Rt_dynamics} with 
\begin{align}
    \widetilde W & = \int_0^\cdot \sign(\overline X^0_t)\sqrt{1-\big\|\rho\big(f(|\overline X_t^0|,X_t),X_t\big)\big\|^2}\,dB^0_t + \int_0^\cdot \rho\big(f(|\overline X_t^0|,X_t),X_t\big)^\top dB_t \label{eq:Wtilde}\\
    \intertext{and}
    \widetilde \Phi & = - \int_0^\cdot \frac{\widehat \sigma(1,X_t)}{\nu(X_t)} dL_t^1(X^0). \label{eq:Phi_tilde}
\end{align} 
Note that the only integer value $X^0$ can take on $[0,\xi)$ is one, which is why the only local time term in \eqref{eq:Phi_tilde} is $L^1(X^0)$. Since Assumption~\ref{ass:coefficients_convex}  ensures that $\widetilde \sigma(r,\cdot)$ is of class $C^2_b$,  
the claimed dynamics for $R$ will follow as in the derivation of Theorem~\ref{thm:interval_Z} (applied locally on the time interval $[0,\xi)$) once we verify that the condition $\sup_{(r,x) \in (0,1] \times \partial D} \|\rho(r,x)\| < 1$ holds.  To this end, note that 
$\rho(r,x) = P(y)v/\|v\|$, where $y = H(r,x)$ and $v = \sigma(y)^\top \nabla \rho_D(y)$.  Set
\begin{equation} \label{eq:unit_normal}
   n(y):=\frac{\sigma^{-1}(y)\gamma(y)}{\|\sigma^{-1}(y)\gamma(y)\|},
\end{equation}
which is a well-defined unit vector since
$\sigma(y)$ is invertible and $a(y) = \sigma(y)\sigma(y)^\top$ is uniformly elliptic. We claim that 
\begin{equation} \label{eq:projection_identity}
    P(y)^\top P(y)=I_d-n(y)n(y)^\top.
\end{equation}Indeed, from \eqref{eq:projection} we have that
$P(y)^\top P(y)$ is the orthogonal projection onto the row space of
$\sigma_X(y) = J_h(y)\sigma(y)$, which is a $(d-1)$-dimensional space due to the degeneracy of $J_h(y)$ in the direction $\gamma(y)$. As such, we see that 
$n(y)$ is a unit vector orthogonal to that row space because 
\[
    J_h(y) \sigma(y)n(y) 
   =\frac{J_h(y)\gamma(y)}{\|\sigma(y)^{-1}\gamma(y)\|}=0,
\]
by \eqref{eq:constancy_condition} and \eqref{eq:flow_oblique}. This establishes the identity \eqref{eq:projection_identity} from which we obtain, with $y = H(r,x)$, the relationship
\begin{align} 
1 - \|\rho(r,x)\|^2 & = 1- \frac{\nabla \rho_D(y)^\top \sigma(y) P(y)^\top P(y)\sigma(y)^\top\nabla \rho_D(y)}{\nabla \rho_D(y)^\top a(y)\nabla \rho_D(y)} \nonumber \\
& = \frac{(\nabla \rho_D(y)^\top \gamma(y))^2}{\|\sigma^{-1}(y)\gamma(y)\|^2\nabla \rho_D(y)^\top a(y)\nabla \rho_D(y)}. \label{eq:rho_norm} 
\end{align}
From \eqref{eq:gamma_uniform_estimate} we see that $\nabla \rho_D(y)^\top \gamma(y) \le -c_0 < 0$ for all $y$ so that the  numerator in \eqref{eq:rho_norm} is bounded away from zero. Moreover, by  boundedness of $\gamma$ and $a$, uniform ellipticity of $\sigma$, and degree-zero homogeneity of $\nabla \rho_D$ we see that the denominator in \eqref{eq:rho_norm} is bounded, establishing uniform boundedness of $\|\rho(\cdot,\cdot)\|$ away from one. As such, the computations leading to the conclusions of Theorem~\ref{thm:interval_Z} are valid (see the proof of Theorem~\ref{thm:halfline_Z} for the analogous computation) establishing that $R$ satisfies \eqref{eq:Rt_dynamics} on $[0,\xi)$.

Next, we turn to verifying the dynamics of $H(R_t,X_t)$. 
First, by differentiating $H$ given by \eqref{eq:H} with respect to $r$ we have that
\[\partial_r H(r,x) = \dot \eta_{\tau(r,x)}(x)\partial_r\tau(r,x) = \gamma\big(H(r,x)\big)\partial_r\tau(r,x),\]
    where we used that $\eta$ satisfies \eqref{eq:flow_oblique}. To obtain an explicit expression for $\partial_r \tau$ we differentiate both sides of the identity $r = \rho_D(H(r,x))$ to obtain
    \[1 = \nabla \rho_D\big(H(r,x)\big)^\top \partial_r H(r,x) = \nabla \rho_D\big(H(r,x)\big)^\top \gamma\big(H(r,x)\big) \partial_r\tau(r,x).\]
    From these expressions we obtain 
    \begin{equation} \label{eq:drH}
        \partial_r \tau(r,x) = \frac{1}{\nabla \rho_D(H(r,x))^\top \gamma(H(r,x))}, \qquad \partial_r H(r,x) = \frac{\gamma(H(r,x))}{\nabla \rho_D(H(r,x))^\top \gamma(H(r,x))}.
    \end{equation}
    Next, from the expression $y  = H(\rho_D(y),h(y))$ for $y \in D^*$ we can differentiate with respect to $y$ to obtain
    \[I_d = \partial_r H\big(\rho_D(y),h(y)\big)\nabla \rho_D(y)^\top + J_H\big(\rho_D(y),h(y)\big)J_h(y),\]
    where $J_h$ is the Jacobian of $h$ and $J_H$ is the Jacobian of $H$ (in the second component). 
    Evaluating at $y = H(r,x)$ for any $(r,x) \in (0,1] \times \partial D$ gives
    \begin{equation} \label{eq:dxH}
        I_d = \partial_r H(r,x)\nabla \rho_D\big(H(r,x)\big)^\top + J_H(r,x)J_h\big(H(r,x)\big).
    \end{equation}
    By taking second derivatives one can obtain an explicit expression for $\partial_{rr}H$ as well as identities that $\partial_{r}J_H$ and the second-derivative tensor of $H$ in the $x$-variable satisfy. Now applying It\^o's formula to $H^i(R,X)$ for $i=1,\dots,d$ yields
    \begin{equation} \label{eq:dH}
    \begin{aligned}
        dH^i(&R_t,X_t)  = \partial_r H^i(R_t,X_t)dR_t + \nabla_x H^i(R_t,X_t)^\top dX_t + \tfrac{1}{2}\partial_{rr}H^i(R_t,X_t)d[R]_t \\
        & \qquad \qquad + \partial_r \nabla_x H^i(R_t,X_t)^\top d[R,X]_t + \tfrac{1}{2}\mathrm{Tr}\big(\nabla^2 H^i(R_t,X_t)d[X]_t\big) \\
        & = \Big(\partial_r H^i(R_t,X_t)\widetilde b(R_t,X_t) + \nabla_x H^i(R_t,X_t)^\top \alpha_X(R_t,X_t) + \tfrac{1}{2}\partial_{rr}H^i(R_t,X_t)\widetilde \sigma^2(R_t,X_t) \\
        & \qquad + \widetilde \sigma(R_t,X_t)\partial_r \nabla_x H^i(R_t,X_t)^\top \sigma_X^P(R_t,X_t)\rho(R_t,X_t) + \tfrac{1}{2}\mathrm{Tr}\big(\nabla^2H^i(R_t,X_t) a_X^P(R_t,X_t)\big)\Big)dt\\
        & \qquad + \partial_r H^i(R_t,X_t)\widetilde \sigma(R_t,X_t)d\widetilde W_t + \nabla_x H^i(R_t,X_t)^\top  \sigma_X^P(R_t,X_t) dB_t   + \partial_r H^i(R_t,X_t)d\widetilde \Phi_t,
    \end{aligned} 
    \end{equation}
    where $a_X^P = \sigma_X^P(\sigma_X^P)^\top$.
    Next, we obtain a more explicit expression for the local martingale part of $H = (H^1,\dots,H^d)$. Working in matrix form, omitting function evaluations for brevity and recalling the definition of $\widetilde \sigma$ and $ \sigma_X^P$ given by \eqref{eq:widetilde_sigma} and \eqref{eq:sigmaX_bar} respectively, as well as the derivative relationships \eqref{eq:drH} and \eqref{eq:dxH} for $H$, we see from \eqref{eq:dH} that the local martingale part is given by
    \begin{align*} \partial_r H \widetilde \sigma d\widetilde W +J_HJ_h\sigma P^\top dB 
    & = \partial_r H\widetilde \sigma  d\widetilde W + (I_d - \partial_r H \nabla \rho_D^\top)\sigma P^\top dB \\
    &  = \sigma\bigg(\frac{\sigma^{-1}\gamma}{\nabla \rho_D^\top \gamma}\widetilde \sigma d\widetilde W + P^\top dB - \frac{\sigma^{-1}\gamma}{\nabla \rho_D^\top \gamma}\nabla \rho_D^\top \sigma P^\top dB\bigg) \\
    & = \sigma\bigg(\frac{\sigma^{-1}\gamma}{\nabla \rho_D^\top \gamma}\sign(\overline X^0)\widetilde \sigma \sqrt{1-\|\rho\|^2}dB^0+P^\top dB\bigg)\\
    & = \sigma dW.
    \end{align*} In the penultimate equality we expanded out the formula for $\widetilde W$ in \eqref{eq:Wtilde}, and in the final equality we recognized that the given expression is precisely $W$ as defined in \eqref{eq:W_convex}. Note that $W$ is indeed a Brownian Motion by L\'evy's characterization, since its quadratic variation is 
    \begin{align*}
        \frac{d[W]_t}{dt} = \frac{\sigma^{-1}\gamma\gamma^\top \sigma^{-T}}{(\nabla \rho_D^\top \gamma)^2}\widetilde \sigma^2(1-\|\rho\|^2) + P^\top P = \frac{\sigma^{-1}\gamma\gamma^\top \sigma^{-T}}{\|\sigma^{-1}\gamma\|^2} + I_d - nn^\top = I_d,
    \end{align*}
    where we used that $1-\|\rho\|^2$ is given by \eqref{eq:rho_norm} and the projection matrix term is given by \eqref{eq:projection} with $n$ given by \eqref{eq:unit_normal}. As such, we see that the dispersion coefficient of $H(R,X)$ is given by $\sigma(H(R,X))$. Similar computations using first- and second-order identities involving derivatives of $H$, which we omit here, show that the drift terms given in \eqref{eq:dH} reduce to $b(H(R,X))$. Finally, for the reflecting term we note from \eqref{eq:Phi_tilde}, \eqref{eq:drH} and  \eqref{eq:dH} that 
    \begin{align*}
        \partial_r H(R_t,X_t)d\widetilde \Phi_t & =  -\frac{\widehat \sigma(1,X_t)\gamma(H(1,X_t))}{\nu(X_t)\nabla \rho_D(H(1,X_t))^\top \gamma(H(1,X_t))}
  dL_t^1(X^0) = d\Phi_t,
    \end{align*}
    where $\Phi$ is given by \eqref{eq:Phi_convex} and we used the fact that $X_t^0 = 1 \implies R_t = 1$ to replace all instances of $R_t$ with one. Note that $\Phi$ satisfies the inward-pointing condition specified in \eqref{eq:Phi_oblique} since $\nabla \rho_D(y)^\top \gamma(y) < 0$ for every $y \in \partial D$ as prescribed by \eqref{eq:gamma_inward}.

    As such, we see that $Y = H(R,X)$ satisfies the RSDE \eqref{eq:RSDE} with oblique reflection \eqref{eq:Phi_oblique} on the time interval $[0,\xi)$. However, by the definition of $H$ in \eqref{eq:H}, the definition of $\tau$ in \eqref{eq:Y_reconstruct} and the fact that the flow $\eta_t(x) \to 0 \iff t \uparrow T(x)$ guaranteed by Lemma~\ref{lem:flow_convex}\ref{item:convex_flow}, we see that $Y_t = 0 \iff R_t = 0$. As such, the explosion time has the representation $\xi = \inf\{t \ge 0: Y_t = 0\}$. By Lemma~\ref{lem:polar} we see that this event has probability zero establishing that $\xi = \infty$, $\P$-a.s. This completes the proof of \ref{item:SDE_system} and also shows that $Y = H(R,X)$ is a global solution to the RSDE \eqref{eq:RSDE} with oblique reflection \eqref{eq:Phi_oblique}, completing the proof of the theorem.

\end{document}